\documentclass[12pt,onecolumn]{IEEEtran}

\usepackage{amssymb}
\usepackage{graphicx,color}
\usepackage{amsthm}
\usepackage{amsmath}
\usepackage{enumerate}
\usepackage{tikz}

\usepackage{float}
\usepackage{subfig}
\usepackage{hyperref}
\usepackage{multirow}
\usepackage{rotating}

\theoremstyle{plain}
\newtheorem{theorem}{Theorem}
\newtheorem{proposition}[theorem]{Proposition}
\newtheorem{corollary}[theorem]{Corollary}
\newtheorem{lemma}[theorem]{Lemma}

\newtheorem{condition}[theorem]{Condition}

\theoremstyle{definition}
\newtheorem{remark}[theorem]{Remark}

\newtheorem{definition}[theorem]{Definition}

\def\ra{\rightarrow}

\def\ba{\begin{array}}
\def\ea{\end{array}}
\def\bi{\begin{itemize}}
\def\ei{\end{itemize}}

\def\mE{\mathbb{E}}
\def\m1{1}

\definecolor{amber(sae/ece)}{rgb}{1.0, 0.49, 0.0}
\definecolor{amethyst}{rgb}{0.6, 0.4, 0.8}
\definecolor{antiquefuchsia}{rgb}{0.57, 0.36, 0.51}
\definecolor{auburn}{rgb}{0.43, 0.21, 0.1}
\definecolor{cadmiumgreen}{rgb}{0.0, 0.42, 0.24}
\definecolor{byzantium}{rgb}{0.44, 0.16, 0.39}
\definecolor{burntsienna}{rgb}{0.91, 0.45, 0.32}

\begin{document}
\title{Tight bounds on the convergence rate of \\ generalized ratio consensus algorithms}
\date{\today}

\author{Bal\'azs Gerencs\'er\thanks{B. Gerencs\'er is with the Alfr\'ed R\'enyi Institute of Mathematics, Budapest, Hungary and E\"otv\"os Lor\'and University, Department of Probability and Statistics, Budapest, Hungary. {\tt\small gerencser.balazs@renyi.hu}
    He was supported by NKFIH (National Research, Development and Innovation Office) Grants PD 121107 and KH 126505.}, %
  L\'aszl\'o Gerencs\'er\thanks{L. Gerencs\'er is with SZTAKI, Institute for Computer Science and Control, Budapest, Hungary. {\tt\small gerencser.laszlo@sztaki.hu.}}} 

\maketitle

\begin{abstract}
 The problems discussed in this paper are motivated by general ratio consensus algorithms, introduced by Kempe, Dobra, and Gehrke (2003) in a simple form as the push-sum algorithm, later extended by B\'en\'ezit et al.\ (2010) under the name weighted gossip algorithm. We consider a communication protocol described by a strictly stationary, ergodic, sequentially primitive sequence of non-negative matrices, applied iteratively to a pair of fixed initial vectors, the components of which are called values and weights defined at the nodes of a network. The subject of ratio consensus problems is to study the asymptotic properties of ratios of values and weights at each node, expecting convergence to the same limit for all nodes. The main results of the paper provide upper bounds for the rate of the almost sure exponential convergence in terms of the spectral gap associated with the given sequence of random matrices. It will be shown that these upper bounds are sharp. Our results complement previous results of Picci and Taylor (2013) and Iutzeler, Ciblat and Hachem (2013).
\end{abstract}


\section{Introduction}
\label{sec:Introduction}

\subsection{The setup and the ratio consensus algorithm} The problems discussed in this paper are motivated by the study of general ratio consensus algorithms, introduced in \cite{kempe2003gossip} in a simple form as the push-sum algorithm, and later extended in \cite{benezit2010weighted} under the name weighted gossip algorithm for solving a class of distributed computation problems. The algorithm is designed to solve a consensus problem over a network of agents, based on asynchronous communication. The objective of the consensus can be expressed in its simplest way as to achieve the average of certain values given at each node. The original problem formulation and the algorithm has been adapted to model a number of real-life situations such as platooning, sensor networks  or smart grids, see  \cite{dimakis2006geographic}, \cite{hadjicostis2011asynchronous}. 

Various relaxations and extensions of the baseline model were proposed in the literature. A nice application of the push-sum algorithm for computing the eigenvectors of a large symmetric matrix, corresponding to the adjacency matrix of an undirected graph, was given in \cite{kempe2008decentralized}. Another application is distributed convex optimization, see \cite{tsianos2012push}. A general class of solvable consensus problems for distributed function computation was introduced in \cite{hendrickx2011distributed}.

The basic setup for this class of methods is a communication network represented by a directed graph $G=(V,E),$ to each node $i$ of which a pair of real numbers $x^i$ and $w^i\ge 0$ is associated, such that not all of the $w^i$-s are $0$. They are often called the values and the weights. The problem is then to compute the ratio $\sum_i x^i/\sum_i  w^i,$ at all nodes, using only local interactions allowed by $G=(V,E)$ in an asynchronous manner. In the special case when $w^i=1$ for all nodes the problem reduces to the average consensus problem.

A convenient illustration of the above problem is the following: $x^i$ unit of some chemical is dissolved in a solvent of $w^i \ge 0$ units leading to a solution with concentration $x^i/w^i$ at node $i.$ The problem equivalent to the one above is then to compute the concentration of the grand total, defined as $\sum_i x^i/\sum_i w^i,$ using only local transfers allowed by $G=(V,E)$ in an asynchronous manner.

Let $|V|=p$ and let $x_0 = x =(x^1, \ldots, x^p)^\top$ and $w_0 = w =(w^1, \ldots, w^p)^\top$ denote the vectors of initial values and weights, respectively, at time $0,$ assuming $w \ge 0, w \neq 0.$ We update both the values and weights successively as follows. Let $x_{n-1}$ and $w_{n-1}$ denote the $p$-vector of values and weights, respectively, at time $n-1.$ Select a directed edge $f_n=(i,j) \in E$ randomly, representing the communicating pair at time $n$. Then the sender, node $i,$ initiates a transactions by sending a fraction, say $\alpha_{ji}$ with $0 < \alpha_{ji} < 1,$ of his/her values and weights to the receiver, node $j.$ It is initially  assumed that the sequence of edges $(f_n)$ is i.i.d., with the probability of choosing an edge $f=(i,j)$ being denoted by $q_{ij}.$

In the context of the above illustration via elementary chemistry the algorithm is equivalent to mixing a fraction of the solution at node $i$ into the current solution at node $j.$ It is then expected that in the limit we get solutions with identical concentrations at each node.

The above algorithm, when setting $\alpha_{ji}= 1/2$ for all edges, is the celebrated push-sum method. The dynamics of the algorithm can be formally described by the equations 
\begin{equation}
x_n = A_n x_{n-1} \quad {\rm and} \quad w_n = A_n w_{n-1} 
\end{equation}
for $n \ge 1$, where $A_n$ is a $p \times p$ random matrix obtained from the identity matrix by modifying its $i$-th column as follows: 
\begin{equation}
\label{eq:def An PS One Column}
A_n^{ii} = 1 - \alpha_{ji}  \quad A_n^{ji} = \alpha_{ji} \quad  A_n^{ki}=0 \quad {\rm for}~ k \neq i,j.
\end{equation}

The above problem can be modified by allowing packet losses, see [3]. When a packet loss occurs along the edge from $i$ to $j,$ denoted by $(j,i),$ the content of node $j$ is not changed. Packet losses are assumed to occur randomly and independently. The functionality of the network at time $n$ is described by a collection of indicators $\rho_n(f), ~f \in E$: $\rho_n(f)=1$ if the edge $f$ fails at time $n$, otherwise $\rho_n(f)=0.$ The probability of failure along edge $f$ is $0 \le r_{f} <1$ at any time, so that $P (\rho_n(f) = 1)= r_{f}.$ With these notations, assuming $f_n=(j,i),$ the matrix $A_n$ will have the following structure with a single, possibly non-zero off-diagonal element in the positions $(j,i)$: 

\begin{equation}
	\label{eq:def An PS with Loss}
	\left(
	\begin{array}{cccccccc}
	1 & 0 &\cdots&&&   &&  0\\
	0 & 1 &\cdots&&&   &&  0\\
	\vdots & \vdots & \ddots \\
	&&&1- \alpha_{ji}&&0\\
	&&&&\ddots\\
	&&& (1 -\rho_n(f_n)) \alpha_{ji}&&1\\
	&&&&&&\ddots&\vdots\\
	0&&&  &&&\cdots&1\\
	\end{array}
	\right).
\end{equation}

We note in passing that the coordinates of vectors and the elements of matrices will be indicated by superscripts, while their dependence on the discrete time $n$ will be indicated by subscripts.

\subsection{A generalized framework.}The above form of the push-sum or weighted gossip algorithm has a natural extension reflecting the possibility of certain schedules  in choosing the sequence of interacting pairs of agents, as in the case of geographic gossip, randomized path averaging or one-way averaging, \cite{dimakis2006geographic, benezit2010order, benezit2010weighted}. 

\medskip
In addition, we may consider a significantly broader class of matrices, allowing much more complex network dynamics. Technically speaking, we consider a strictly stationary, ergodic sequence of $p \times p$ random matrices with non-negative entries $A_1, A_2,\ldots.$ Let $x, w \in \mathbb R^p$ denote a pair of initial vectors, such that $w \ge 0, w \neq 0$. Our objective is to study the asymptotic properties of the ratios 
\begin{equation}
e_i^\top A_n A_{n-1} \cdots A_1 x/ e_i^\top A_n A_{n-1} \cdots A_1 w, \quad i=1,...,p
\end{equation}
where $e_i$ is the unit vector with a single $1$ in its $i$-th coordinate. 

For a start we provide a brief summary of two classical results on products of strictly stationary, ergodic sequences of random matrices, and recapitulate and extend a relevant application as Theorem \ref{thm:barMnx vs barMnw}.
The key results of this paper are stated as Theorems \ref{thm:ei Mnx vs ei Mnw IID}, \ref{thm:ei Mnx vs ei Mnw STATIONARY BOUNDED}, \ref{thm:ei Mnx vs ei Mnw STATIONARY GEN} ~and ~\ref{thm:ei Mnx vs ei Mnw COLUMN STOCH}, extending previous results on the almost sure exponential convergence in the context of ratio consensus such as given in  \cite{iutzeler2013analysis} and \cite{gerencser2018push}, in particular providing upper bounds for the almost sure exponential convergence rate in terms of spectral gaps associated with stationary sequences of matrices. It will be shown that these upper bounds are sharp in Theorem \ref{thm:max i of ratios via barMnx vs barMnw}, thus solving an open problem formulated in the conclusion of one of the fundamental papers \cite{benezit2010weighted} under very general conditions, quoting from their Conclusion:

\medskip
"The next step of this work is to compute analytically the speed of convergence of Weighted Gossip. In classical Gossip, double stochasticity would greatly simplify derivations, but this feature disappears in Weighted Gossip, which makes the problem more difficult."

\medskip
The proofs are based on the careful analysis of random products $M_n =  A_n A_{n-1} \cdots A_1$ for random sequence of non-negative matrices using Oseledec's theorem. 
The application of results in the theory of products of random matrices in the context of consensus algorithms was previously initiated and elaborated in \cite{picci2013almost} for the case of linear gossip algorithms with pairwise, bidirectional, symmetric communication.  
While we rely partially on the same mathematical methodology, the range of communication protocols that we consider is significantly broader, in particular we consider weighted gossip algorithms.

Our work complements and extends the result of \cite{iutzeler2013analysis} in which an upper bound for the rate (or the exponent) of almost sure exponential convergence of a (sampled) weighted gossip algorithms was derived.

The paper is organized as follows: Sections 
\ref{sec:Technical preliminaries} -- \ref{sec:Push Sum} are devoted to  the description of the subject matter and the main results of the paper with minimal technical details: starting with two sections presenting a few preliminary technicalities, a section on normalized products, a section with the statements and interpretations of the main results, followed by a brief section on push-sum algorithms. In the last two sections of the main body of the paper we elaborate on the major mathematical details: in Section \ref{sec:Proofs of Main Theorems} we describe the  essential fabric of the proofs of the main theorems,  while in Section \ref{sec:Gap vs Birkhoff} an interlude on the connection between spectral gap and Birkhoff's contraction coefficient is added.  A brief discussion and conclusion wraps up the material of the main body of the paper. Relevant, but minor technical details will be given in the Appendices. Altogether we intend to give a self-contained presentation of the subject matter and of the background material.

\section{Technical preliminaries}
\label{sec:Technical preliminaries}

For the formulation of our results we recall two basic facts on the product of random matrices.

\begin{proposition} [F\"urstenberg and Kesten's theorem, \cite{furstenberg1960products}]
\label{prop:FK}
Let $A_1, A_2,\ldots,$ be a strictly stationary, ergodic process of $p \times p$ random matrices over a complete probability space $(\Omega, \cal F, P)$ such that $\mE \log^+ \Vert A_1 \Vert < \infty.$ Then the almost sure limit 

\begin{equation}
\label{eq:def top Lyap exp via E}
\lambda_1 = \lim_{n \rightarrow \infty} {\frac 1 n} \log \Vert A_n A_{n-1} \cdots A_1 \Vert < \infty
\end{equation} 
exists and it is equal to  
\begin{align}
\label{eq:def top Lyap exp via E}
\lim_{n \rightarrow \infty} &{\frac 1 n} \mE \log \Vert A_n A_{n-1} \cdots A_1 \Vert \nonumber \\
= \inf_n &{\frac 1 n} \mE \log \Vert A_n A_{n-1} \cdots A_1 \Vert.
\end{align}
Note that we may have $\lambda_1 = - \infty.$ 
\end{proposition}

A more refined asymptotic characterization of $ A_n A_{n-1} \cdots A_1$ is given by Oseledec's theorem. To appreciate the novelty and power of this theorem we make a brief elementary detour in the field of Lyapunov exponents, see \cite{jungers2009joint}. Let $(A_n), n \ge 1$ be a fixed sequence of $p \times p$ matrices. For any $x \in \mathbb R^p$ define the Lyapunov exponent of $x$ with respect to (w.r.t.) $(A_n)$ as
$$
\lambda(x) := \limsup_{n \rightarrow \infty} {\frac 1 n} \log \vert A_n A_{n-1} \cdots A_1 x \vert.
$$
Next, for any extended real number $-\infty \le \mu \le +\infty$ define the set
	\begin{equation}
	\label{eq:def L mu}
	L_{\mu} = \{x \in \mathbb R^p: \lambda(x) \le \mu \}.	
	\end{equation}
It is easily seen that $L_{\mu}$ is a linear subspace of $\mathbb R^p$ and for $\mu < \mu'$ we have $L_{\mu}\subseteq L_{\mu'}.$ It is also readily seen that $L_{\mu}$ is continuous  from the right: if $x \in L_{\mu_j}$ for a sequence of $\mu_j$-s such that $\mu_j$ tend to $\mu$ from above, then we have also $x \in L_{\mu}.$ Since there can be only a finite number of strictly descending subspaces it follows that there is a finite number of possible values of the Lyapunov exponents, $ +\infty \ge \mu_1 > \mu_2 > \ldots > \mu_q \ge -\infty,$ such that 
\begin{equation}
\label{eq:def L mu r}
\mathbb R^p = L_{\mu_1} \supsetneq L_{\mu_2} \ldots \supsetneq   L_{\mu_q} \supsetneq \{0\} =: L_{\mu_{q+1}},
\end{equation}
where $ L_{\mu}$ is a piecewise constant function of $\mu$ with points of discontinuity exactly at $\mu_i.$ Thus for $\mu_{r-1}> \mu \ge \mu_{r}$ we have $L_{\mu} = L_{\mu_r}$ for $2 \le r \le q$ and for $\mu_{q}> \mu$ we have $L_{\mu} = \{0\}.$ It follows that for $1 \le r \le q$ 
\begin{equation}
\label{eq: x in setminus L mu} 
x \in  L_{\mu_r}  \setminus L_{\mu_{r+1}} \quad {\rm implies} \quad \lambda(x)= \mu_r.
\end{equation}
Let the dimension of $L_{\mu_r}$ be denoted by $i_r$, with $1 \le r \le q+1$ (with $i_{q+1}=0$). Then the co-dimension of $L_{\mu_r}$ relative to $ L_{\mu_{r+1}}$ is $i_r - i_{r+1}$, which can be interpreted as the multiplicity of the Lyapunov exponent $\mu_r.$ Accordingly, we define the full spectrum of Lyapunov exponents $\lambda_1 \ge \lambda_2 \ge \ldots \ge \lambda_p,$ allowing the values $\pm \infty,$ is obtained by setting for $1 \le i \le p$ 
\begin{equation}
\label{eq:lambda i eq mu r} 
\lambda_i = \mu_r \quad {\rm if} \quad i_r \ge i > i_{r+1}.
\end{equation}

If $(A_n)=(A_n(\omega))$ is the realization of a strictly stationary ergodic process then the above observations can be extended to the following fascinating result, stated first in \cite{oseledec1968multiplicative}, and restated and proved under weaker condition in \cite{raghunathan-oseledec}: 
	
\begin{proposition} [Oseledec's theorem] Assume that $(A_n)$ is a strictly stationary ergodic process of $p \times p$ matrices such that $\mE \log \Vert A_1 \Vert^+ < \infty.$ Then there exists a subset $\Omega' \subset \Omega$ with $P(\Omega')=1$ such that for all $\omega \in \Omega'$ and for any $x \in \mathbb R^p$ the limit below exists:
\begin{equation}
\lambda(x) =\lim_{n \rightarrow \infty} {\frac 1 n} \log \vert A_n A_{n-1} \cdots A_1 x\vert. 
\end{equation}	
Moreover the Lyapunov exponents $\lambda_1 \ge \lambda_2 \ge \ldots \ge \lambda_p,$ possibly taking the value $- \infty,$ do not depend on $\omega \in \Omega'.$ Accordingly, $\mu_r$ and $i_r$ for $1 \le r \le q$ do not depend on $\omega \in \Omega'$ either. The mapping $\omega \mapsto   L_{\mu_r}(\omega)$ is measurable from $\Omega$ to the Grassmanian manifold of linear subspaces of dimension $i_r.$ 
In addition, the following almost sure limit exists:
\begin{equation}
\label{eq:Oseledec Lim MnT Mn Root} 
M^* = \lim \left( M_n^T M_n \right)^{1/2n}.
\end{equation} 
\end{proposition}

From the proof given in \cite{raghunathan-oseledec} it follows that taking a singular value decomposition of $M_n := A_n A_{n-1} \cdots A_1$  
\begin{equation}
\label{eq:Mn SVD}
M_n = U_n \Sigma_n V_n,
\end{equation}
where $U_n, V_n$ are orthonormal matrices, and $\Sigma_n$ is diagonal
with entries $\sigma^{1}_{n} \ge \sigma^{2}_{n} ...\ge \sigma^{p}_{n} \ge 0,$ we have 
\begin{equation}
\lambda_k = \lim_{n \rightarrow \infty} {\frac 1 n} \log \sigma^{k}_{n} \quad {\rm a.s.}\quad k=1,...,p. 
\end{equation}
Therefore we have, with $o(1)$ denoting a sequence of random variables tending to $0$ a.s. (almost surely) as $n$ tends to $\infty,$

\begin{equation}
\label{eq:Mn SVD Sigma n}
\Sigma_n = {\rm diag} (e^{(\lambda_k + o(1))n}).
\end{equation}

Surprisingly, the orthonormal matrices $V_n$ will also converge a.s. in a restricted sense. Allowing the possibility of multiplicity of Lyapunov-exponents consider a fixed $\mu_r$ and define
$I_r= \{i: \lambda_i = \mu_r \},$ and let $SV_{n}^{I_r \cdot}$ denote the subspace spanned by the rows of $V_n$ with indices in $I_r.$ Then we have a.s. $\lim ~SV_{n}^{I_r \cdot} = SV^{I_r \cdot}$ for some random subspace $SV^{I_r \cdot}.$ We note in passing that this technical result immediately implies the existence of the a.s. limit in \eqref{eq:Oseledec Lim MnT Mn Root}.

In particular, if $\lambda_1 > \lambda_2,$ then for the first row of $V_n,$ denoted by $v_n^{1 \cdot}$ we have 
\begin{equation}
\lim v_n^{1 \cdot} = v^{1 \cdot}  
\end{equation}
a.s., for some random $v^{1 \cdot}.$ In fact, Ragunathan proved in Lemma 5 of \cite{raghunathan-oseledec} that for any $\varepsilon > 0$ 

\begin{equation}
\label{eq:V1dot_n vs V1dot}
v_n^{1 \cdot} - v^{1 \cdot} = O(e^{-(\lambda_1 - \lambda_2 + o(1))n})  \quad {\rm a.s.}
\end{equation}
Writing 
\begin{equation}
M_n =  u_{n}^{\cdot 1} v^{1 \cdot} \sigma^{1}_{n} + \sum_{k=2}^p  u_{n}^{\cdot k} v^{k \cdot} \sigma^{k}_{n},
\end{equation}
it follows by straightforward calculations that 
\begin{equation}
\label{eq:Mn SVD Rank 1}
M_n = u_{n}^{\cdot 1} v^{1 \cdot} \sigma^{1}_{n} + O(e^{(\lambda_2 + o(1))n}).  
\end{equation}
A rank-$1$ approximation for the product of an strictly stationary, ergodic sequence of column stochastic matrices has been derived in Theorem 3, \cite{tahbaz2010consensus} using different techniques. A deterministic alternative, with exponential rate of convergence, is implied by Proposition 1,  \cite{liu2012asynchronous}.

A nice corollary of Oseledec's theorem, obtained by a straightforward application of Fubini's theorem, is that for all $x \in \mathbb R^p,$ except for a set of Lebesgue-measure zero, we have  
\begin{equation}
\label{eq:lim An A1 x}
\lambda_1 = \lim {\frac 1 n} \log \vert A_n A_{n-1} \cdots A_1 x\vert \quad {\rm a.s.}
\end{equation}

In the special case when $A_n=A$ for all $n,$ arranging the eigenvalues of $A,$ say $\nu_i,$ according to their absolute values in  non-increasing order, we have $\lambda_i = \log | \nu_i |.$

\section{Sequentially primitive non-negative matrix processes}

In the next section we present the extension of a result of \cite{atar1997lyapunov} on the asymptotic behavior of normalized products  
\begin{equation}
  A_n A_{n-1} \cdots A_1 x/ \mathbf 1^T A_n A_{n-1} \cdots A_1 x, 
\end{equation} 
where $\mathbf 1$ is a $p$-vector all coordinates of which are $1$. 
For the generalization of Theorem 1 of \cite{atar1997lyapunov} the extension of the notion of primitivity for a class of matrices and stochastic processes will be needed. For a nice introduction and motivation on this topic see \cite{protasov2012sets}. 
  
Let ${\cal A} = \{ A_1, \cdots ,A_m \}$ be a finite family of $p \times p$ matrices with non-negative entries. We may then ask if there is a product of these matrices (with repetitions permitted) which is strictly positive? The following definition is essentially  given in \cite{protasov2012sets}:

\begin{definition}
A family $\cal A$ of nonnegative $p \times p$-matrices is called primitive if there is at least one strictly positive product of matrices of this family.
\end{definition}

Let $A^0 := \gamma (A)$ denote the $(0,1)$ matrix having a $1$ in a position exactly if in that position $A$ has a positive element. Define the set of matrices ${\cal A}^0 = \{\gamma (A): A \in \cal A\}.$ Then, obviously, $\cal A$ is primitive if and only if ${\cal A}^0$ is primitive. The definition and claim extends to infinite sets of matrices $\cal A.$ 

We will now extend the definition to stationary processes of non-negative random matrices. A matrix is called allowable, if it has no zero row or zero column. It is called row-allowable if it has no zero row.

\begin{definition}
A strictly stationary process of non-negative allowable random matrices $(A_n), n \ge 1,$ is called (forward) sequentially primitive if $M_\tau = A_\tau A_{\tau-1} \cdots A_1$ is strictly positive for some finite stopping time $\tau$ with probability 1 (w.p.1). For any $n\ge 1$ we define the (forward) index of sequential primitivity as 
\begin{equation}
\label{eq:def index of forward seq prim}
\psi_n  = \min \{\psi \ge 1: A_{n+ \psi -1} A_{n+\psi-2} \cdots A_n >0  \}.
\end{equation} 
\end{definition}

Since by assumption $A_n$ is row-allowable we will have $M_n > 0$ with strict inequality for all $n \ge \psi_1.$ It is also clear that a stationary process of non-negative random matrices $(A_n), n \ge 1,$ is (forward) sequentially primitive if and only if the stochastic process $(A^0_n),  n \ge 1,$ is (forward) sequentially primitive.

The definition extends to two-sided processes. In this case we may also define the concept of backward sequential primitivity, and the index of backward sequential primitivity as
\begin{equation}
\label{eq:def index of backward seq prim}
\rho_n  = \min \{\rho \ge 1: A_{n} A_{n-1} \cdots A_{n-\rho+1} >0  \}.
\end{equation}

\begin{lemma}
	\label{lem:seq prim forward backward}
        A two-sided strictly stationary sequence $(A_n)$ is forward sequentially primitive if and only if it is backward sequentially primitive. Moreover, the indices of forward and backward sequential primitivity, $\psi_n$ and $\rho_n,$ have the same distributions.
\end{lemma}

The point in discussing both forward and backward primitivity will become clear in connection with Theorems \ref{thm:ei Mnx vs ei Mnw STATIONARY BOUNDED} and \ref{thm:ei Mnx vs ei Mnw STATIONARY GEN} below in which the natural assumption is that $(A_n), ~n \ge 1$ is \emph{forward} sequentially primitive, and ${\mathbb E} \psi_1 < \infty.$ However, in the proof we do need to ensure that for a two-sided extension of $(A_n)$ we have  ${\mathbb E} \rho_1 < \infty $.

Consider now the case of an i.i.d.\ sequence $(A_n), n \ge 1.$ 
\begin{remark}
	\label{rem:seq prim vs set prim} Let $(A_n), n \ge 1,$ be an i.i.d.\ sequence. Then it is sequentially primitive if and only if the set below is primitive:
	$$
	{\overline {\cal A}}^0 = \{C: P(\gamma (A_1) = C)  > 0  \}.
	$$
\end{remark}

Obviously, the range of $(\gamma(A_n)), n \ge 1$ is finite. This motivates the assumption in the lemma below.

\begin{lemma}
	\label{lem:sigma n tails iid A n}
	 Consider an i.i.d.\ sequence of non-negative, allowable $p \times p$ matrices $(A_n), ~-\infty < n < \infty$ having a finite range $\cal A,$ which is primitive. Then $\psi_n$ is finite w.p.1, and the tail-probabilities of $\psi_n$ decay geometrically, $P(\psi_n > x) < c \exp(-\alpha x)$ with some $c,\alpha >0.$ Analogous results hold for the indices of backward sequential primitivity $\rho_n.$  
	
\end{lemma}

The almost trivial proof will be given in Appendix I. The above lemma implies that $\mE \psi_n < \infty,$ and since $\psi_n$ has the same distribution for all $n,$ the sequence $\psi_n$ is sub-linear, i.e. $\psi_n = o(n)$ a.s. Obviously, the same holds for the backward indices of sequential primitivity, i.e. $\rho_n = o(n)$ a.s.

\section{Normalized products of non-negative random matrices}
\label{sec:Normalized_Products}

In this section we describe the extension of a nice result of \cite{atar1997lyapunov}, the proof of which inspired the proofs of the main theorems of the present paper.

Let $(A_n)$ be a sequence of allowable $p \times p$ matrices. Let $x, w \in \mathbb R^p$ be component-wise non-negative vectors, written as $x,w \ge 0,$ the set of which will be denoted by $\mathbb R_+^p,$ such that $x,w \neq 0.$ Define the sequences 
\begin{align}
x_n &:= M_n x = A_n A_{n-1} \cdots A_1 x,\\
w_n &:= M_n w = A_n A_{n-1} \cdots A_1 w.
\end{align}
Obviously $x_n$ and $w_n$ are non-negative, and since the $A_n$-s are allowable and $x,w \neq 0$, we have $x_n, w_n \neq 0.$
Therefore we can define 
\begin{equation}
\label{eq:DEF pf BAR x and BAR w}
\bar{x}_n = x_n/(1^\top x_n), \quad \bar{w}_n = w_n/ (1^\top w_n).
\end{equation}
The following result is a straightforward extension of \cite{atar1997lyapunov}. In the theorem $\Vert \bar{x}_n- \bar{w}_n \Vert_{\rm TV} : = {\frac 1 2} \sum_{i=1}^p |\bar{x}_n^i -\bar{w}_n^i|$ denotes the total variation distance of the probability vectors $\bar{x}_n$ and $\bar{w}_n$.

\begin{theorem}
\label{thm:barMnx vs barMnw}
	Assume that $(A_n), ~n \ge 1$ is a strictly stationary, ergodic process of random $p \times p$ matrices such that $\mE \log^+
	\|A_1\| < \infty$. In addition assume that $A_n$ is non-negative and allowable
	for all $n,$ and assume that the process $(A_n)$ is sequentially primitive. Then for all pairs $(x,w) \in \mathbb R_+^p \times \mathbb R_+^p,$ except for a set of Lebesgue-measure zero, it holds that
	$$
	\lim_{n \rightarrow \infty} \frac{1}{n} \log \Vert \bar{x}_n- \bar{w}_n \Vert_{\rm TV} = - (\lambda_1 -
	\lambda_2) \quad {\rm a.s.,}
	$$
where $\lambda_1$ and $\lambda_2$ are the first and second largest Lyapunov-exponents associated with $(A_n).$ In addition, for any fixed pair $(x,w) \in \mathbb  R_+^p \times \mathbb R_+^p$ with strictly positive components with no exception it holds that the above limit exists a.s. and  
$$
\lim_{n \rightarrow \infty} \frac{1}{n} \log \Vert \bar{x}_n- \bar{w}_n \Vert_{\rm TV} \le - (\lambda_1 -
\lambda_2).
$$
\end{theorem}
The proof of Theorem~\ref{thm:barMnx vs barMnw} is a straightforward extension of the proof of Theorem 1 in \cite{atar1997lyapunov} and will be given in Appendix II. We should note, however, that the proof given in \cite{atar1997lyapunov} contains two non-trivial deficiencies. These will be rectified by the lemmas below, the proofs of which will be given also in Appendix II. The first lemma was implicitly stated in \cite{atar1997lyapunov}, with a minor flaw in the proof:  
		
		\begin{lemma} 
			\label{lem:Oseledec for An positive and x positive} 
			 Let the sequence of matrices $(A_n)$ be as in Theorem \ref{thm:barMnx vs barMnw}. Then there exists a subset $\Omega' \subset \Omega$ with $P(\Omega')=1$ such that for all $\omega \in \Omega'$ it holds that any strictly positive vector $x >0, x \in \mathbb R^p$ is contained in $x \in L_{\mu_1}  \setminus L_{\mu_{2}},$ see \eqref{eq:def L mu r} -- \eqref{eq:lambda i eq mu r}: 
			\begin{equation}
			\lambda_1 = \lim {\frac 1 n} \log \vert A_n A_{n-1} \cdots A_1 x\vert.  
			\end{equation}	
		\end{lemma}

The second result was tacitly used in \cite{atar1997lyapunov}, with no proof. Here the notion of exterior product of vectors and matrices, denoted by $x \wedge w$ and $A \wedge B,$ resp., is used. Here $x \wedge w$ can be identified with the anti-symmetric matrix $x w^\top - w x^\top,$ and $(A \wedge B) (x \wedge w) = Ax \wedge Bw,$ see \cite{flanders1963differential}. 

	\begin{lemma}  
		\label{lem:almost all x and w pairs}
		Let $(A_n)$ be a strictly stationary, ergodic process of $p \times p$ random matrices $A_1, A_2,\ldots,$ such that $\mE \log^+ \Vert A_n \Vert < \infty.$ Consider the exterior product space $\mathbb R^p \wedge \mathbb R^p$ and the matrices $A_n \wedge A_n$ acting on it. Then for all pairs $(x,w) \in \mathbb R^p \times \mathbb R^p$, except for a set of Lebesgues measure zero, the a.s. limit
		\begin{equation*}
		\lim_{n \rightarrow \infty} \frac{1}{n} \log \vert ((A_nA_{n-1}\cdots A_1) \wedge (A_nA_{n-1}\cdots
		A_1))(x \wedge w) \vert 
		\end{equation*}
		exists and is equal to $\lambda_1 + \lambda_2.$ 
\end{lemma}

Motivated by Theorem \ref{thm:barMnx vs barMnw} we consider the possibility of an extension of the results concerning the push-sum or weighted gosspip algorithms under significantly more general conditions.

\section{A generalized ratio consensus}
\label{sec:Main Theorems}

In this section we will formalize our main results on the convergence rate of a generalized ratio consensus algorithm. The common setup for our results will be based on Theorem \ref{thm:barMnx vs barMnw}. However, this will have to be complemented by a variety of additional conditions imposed on $(A_n)$.

For the formulation of our technical results we will need to impose further conditions on the positive elements of $A_n,$ controlling the possibility of moving a random fraction (or share) of values and weights during a transaction. Let us introduce the following notations for the minimal and maximal positive elements of $A_n$:
\begin{equation}
\label{eq:def alpha n beta n}
\alpha_n: = \min_{ij} \{A_{n}^{ ij}: A_{n}^{ ij} >0 \}, \qquad \beta_n := \max_{ij} A_{n}^{ ij}.
\end{equation}
Since $\beta_n $ is equivalent to $\|A_n\|,$ it follows immediately that 
$\mE \log^+ \beta_n < \infty$. A direct consequence of this is that for any $\varepsilon >0$ we have a.s. $\beta_n = O(e^{\varepsilon n}),$ i.e. $\beta_n$ is sub-exponential (see below). A twin pair of the condition $\mE \log^+ \beta_n < \infty$ is the following:

\begin{condition}
\label{cond:E log alpha n} 
Let $(A_n)$ be a strictly stationary, ergodic process of random, $p \times p$ non-negative matrices. We assume that ${\mE} \log^- \alpha_n > -\infty,$ where $\alpha_n$ is the minimal positive element of $A_n$ defined above.   	
\end{condition}

A direct consequence of this condition is that ${\mE} \log^+ \frac 1 {\alpha_n} < \infty,$ implying that $\frac 1 {\alpha_n}$ is sub-exponential. The above condition is obviously satisfied if $(A_n)$ takes its values form a finite set, say $\cal A,$ w.p.1, which is the case with the push-sum algorithm allowing packet loss.

\begin{theorem}
	\label{thm:ei Mnx vs ei Mnw IID}
	Assume that the conditions of Theorem~\ref{thm:barMnx vs barMnw} are satisfied, in addition the sequence $(A_n)$ is i.i.d., and $ \lambda_1 - \lambda_2 > 0.$ Furthermore, assume that the minimal positive elements of $A_n$ satisfy Condition \ref{cond:E log alpha n}. Let $e_k$ denote the $k$-th unit vector for any $k=1, \dots ,p.$ 
	Take an arbitrary vector of initial values $x \in  \mathbb R^p,$ and a non-negative vector of initial weights $w \in \mathbb R^p_+ $ such that $w \neq 0.$
	Then ratio consensus takes place and an explicit upper bound for the rate of convergence can be given as follows: for all $i=1, \ldots, p$ we have 
	\begin{equation}
	\label{eq:ek Mnx vs ek Mnw}
	\limsup_{n \rightarrow \infty} {\frac 1 n} \log \left \vert \frac{e_i^\top M_n x}{e_i^\top M_n w}  - \frac{v^{1\cdot}x}{v^{1\cdot}w} \right \vert \le -(\lambda_1 - \lambda_2) \quad {\rm a.s.}
	\end{equation}
\end{theorem}

By Theorem \ref{thm:ei Mnx vs ei Mnw IID} for all agents $i$ the values $x_{n}^{i}/w_{n}^{i}$ will converge to the same limit $\pi^T x$ a.s., where $\pi$ is the random vector defined by $\pi={v^{1\cdot}}/{v^{1\cdot} {w}}$, with at least the given rate. 
The limit is random, in contrast to the case of classic push-sum or weighted gossip algorithms without packet loss. On the other hand, there is ample empirical evidence that \emph{decreasing the probability of packet loss} leads to higher concentration of the distribution of $ \pi^T x,$ around $\bar x,$ see \cite{gerencser2018push}.

An extension of the above scenario is obtained if the communicating pairs of agents are chosen according to some time-homogeneous random pattern, which may be different from an i.i.d.\ choice, see geographic gossip, randomized path averaging or one-way averaging, \cite{dimakis2006geographic, benezit2010order, benezit2010weighted}. Thus we come to consider the case when $(A_n)$ is a general, \emph{strictly stationary ergodic} sequence $(A_n)$. As for the additional conditions to be imposed we consider two levels of complexity.

\begin{condition}
\label{cond:BOUNDED ALPHA n BETA n} Let $(A_n)$ be a strictly stationary, ergodic process of random, $p \times p$ non-negative matrices. We say that $(A_n)$ is bounded from below and from above, if there exist $\alpha, \beta > 0$ such that, with the notations of \eqref{eq:def alpha n beta n}, we have a.s.	
\begin{equation}
\label{eq:bounded alpha n beta n}
\alpha_n  \ge \alpha >0, \qquad \beta_n  \le \beta.
\end{equation}
\end{condition}

Again, the above condition is obviously satisfied if the range of $(A_n)$, denoted above by $\cal A,$ is finite.

\begin{theorem}
\label{thm:ei Mnx vs ei Mnw STATIONARY BOUNDED}
Assume that the conditions of Theorem~\ref{thm:barMnx vs barMnw} are satisfied, $ \lambda_1 - \lambda_2 > 0, $ and for the forward index of sequential primitivity $\psi_n$ we have $\mE \psi_n < \infty.$ Furthermore, assume that the positive elements of $A_n$ are bounded from below and from above in the sense of Condition \ref{cond:BOUNDED ALPHA n BETA n}.
Then for any vector of initial values $x \in  \mathbb R^p,$ and any non-negative vector of initial weights $w \in \mathbb R^p_+ $ such that $w \neq 0$ ratio consensus takes place, in fact (\ref{eq:ek Mnx vs ek Mnw}) holds.
\end{theorem}

A further extension of the above result is obtained if the elements of $A_n$ are \emph{not bounded} from above and from below, thus allowing for the possibility of moving a random fraction of values and weights. In this case we need an extra technical condition ensuring some kind of mixing of the process $(A_n).$

\begin{condition}
\label{cond:M MIXING} 
A two-sided strictly stationary process $(\xi_n)$ satisfies a $q$-th order $M$-mixing condition, with $q \ge 1,$ if $\mE |\xi_n|^q < \infty,$ and for any positive integer $N$ we have, with some constant $C>0,$
\begin{equation}
\mE  \Big \vert \sum_{n=1}^N (\xi_n - \mE \xi_n )  \Big \vert^q \le C N^{q/2}.
\end{equation}

\end{condition}

\begin{theorem}
	\label{thm:ei Mnx vs ei Mnw STATIONARY GEN}
	Assume that the conditions of Theorem~\ref{thm:barMnx vs barMnw} are satisfied, $ \lambda_1 - \lambda_2 > 0,$ and for the index of forward sequential primitivity $\psi_n$ we have $\mE \psi_n < \infty.$ Furthermore, assume that $
	a_n= \log \alpha_n$ and $b_n = \log \beta_n$ satisfy a $q$-th order $M$-mixing condition, given in Condition \ref{cond:M MIXING}, with some $q>4$.
	Then for any vector of initial values $x \in  \mathbb R^p,$ and any non-negative vector of initial weights $w \in \mathbb R^p_+ $ such that $w \neq 0$ ratio consensus takes place, in fact (\ref{eq:ek Mnx vs ek Mnw}) holds.
\end{theorem}

It may be of interest to consider an estimate of the average at any time $n$ by taking a weighted average of the respective values of $x_{n}^{i}$ and $w_{n}^{i}.$ In this case Theorems \ref{thm:ei Mnx vs ei Mnw IID}, \ref{thm:ei Mnx vs ei Mnw STATIONARY BOUNDED}, \ref{thm:ei Mnx vs ei Mnw STATIONARY GEN} easily generalize to the following:

\begin{corollary}
		\label{cor:q Mnx vs q Mnw}
		Let $q \in \mathbb R^p_+, q \neq 0$ be a non-negative weight vector. 
		Assume that any of the sets of condtions of Theorems \ref{thm:ei Mnx vs ei Mnw IID}, \ref{thm:ei Mnx vs ei Mnw STATIONARY BOUNDED} or \ref{thm:ei Mnx vs ei Mnw STATIONARY GEN} is satisfied. 
		Then for any vector of initial values $x \in  \mathbb R^p,$ and any non-negative vector of initial weights $w \in \mathbb R^p_+ $ such that $w \neq 0$ we have
		\begin{equation}
		\limsup_{n \rightarrow \infty} {\frac 1 n} \log \left \vert \frac{q^\top M_n x}{q^\top M_n w}  - \frac{v^{1\cdot}x}{v^{1\cdot}w} \right \vert \le -(\lambda_1 - \lambda_2) \quad {\rm a.s.}
		\end{equation}
\end{corollary}

	\begin{proof} [Proof of Corollary \ref{cor:q Mnx vs q Mnw}.] The claim is obtained by a direct and standard convexity argument, see \cite{benezit2010weighted}: for any pair of vectors $a,b \in \mathbb R^p$ such that $b > 0 $ we have 
		\begin{equation}
		\min_i \frac {a_i}{b_i} \le \frac{q^\top a}{q^\top b} \le \max_i \frac {a_i}{b_i}.
		\end{equation}	
		Indeed, this follows from 
		\begin{equation}
		\frac{q^\top a}{q^\top b} = \frac {\sum_i q_i a_i}  {\sum_i q_i b_i} =
		\sum_{i} \left ({\frac {a_i} {b_i}}\right) ~\frac { q_i b_i }  {\sum_j q_j b_j}. 
		\end{equation}
		Setting $a_i=e_i^\top M_n x$ and $b_i=e_i^\top M_nw$ we get 
		\begin{equation}
		\label{eq:min vs max conv}
		\min_i \frac {e_i^\top M_n x}{e_i^\top M_n w} \le \frac{q^\top M_n x}{q^\top M_n w} \le \max_i \frac {e_i^\top M_n x}{e_i^\top M_n w},
		\end{equation}		
		from which the claim follows by Theorems \ref{thm:ei Mnx vs ei Mnw IID}, \ref{thm:ei Mnx vs ei Mnw STATIONARY BOUNDED}, \ref{thm:ei Mnx vs ei Mnw STATIONARY GEN}.
	 	\end{proof}

Let the l.h.s. and the r.h.s. of \eqref{eq:min vs max conv} be denoted by $y_n$ and $z_n,$ respectively. The elementary lemma below, which will be used later on, has been established in \cite{gerencser2018push} for the case of the push-sum algorithm with packet loss:

\begin{lemma}
	\label{lem:min ratio and max ratio monotone}
	The values $y_n$ and $z_n$ are monotone non-decreasing and non-increasing, respectively. In particular, it follows that for any time $n$ we have
		\begin{equation}
		\label{eq:min vs max limit}
		\min_i \frac {e_i^\top M_n x} {e_i^\top M_n w} \le \frac{v^{1\cdot}x}{v^{1\cdot}w} \le  \max_i \frac {e_i^\top M_n x}{e_i^\top M_n w} \qquad {\rm a.s.}
		\end{equation}
\end{lemma}

\begin{proof} [Proof of Lemma \ref{lem:min ratio and max ratio monotone}]
		Indeed, for any index $j$ write 
		\begin{equation}
		h_{n+1,j} := \frac{e_j^\top M_{n+1} x}{e_j^\top M_{n+1} w} = \frac{e_j^\top A_{n+1} M_{n} x}{e_j^\top A_{n+1} M_{n} w} = \frac{q_j^\top M_{n} x}{q_j^\top M_{n} w}
		\end{equation}
		with $q_j^\top=e_j^\top A_{n+1}.$ Since $A_{n+1}$ is non-negative and allowable, we have $q_j \ge 0, q_j \neq 0.$ Thus we get by \eqref{eq:min vs max conv} the inequality $y_n \le h_{n+1,j} \le z_n $ for all $j$ from which the first claim follows. The second claim follows trivially from the established monotonicity, and the fact that, according to Theorem \ref{thm:ei Mnx vs ei Mnw IID}, we have a.s.
		\begin{equation*}
		\lim_{n \rightarrow \infty} \min_i \frac {e_i^\top M_n x} {e_i^\top M_n w} =  \frac{v^{1\cdot}x}{v^{1\cdot}w} = \lim_{n \rightarrow \infty} \max_i \frac {e_i^\top M_n x}{e_i^\top M_n w}. 
		\end{equation*}
\end{proof}

In the special case when $A_n$ is column stochastic for all $n$, as in the case of the push-sum or weighted gossip algorithm with no packet loss, $M_n$ will be column-stochastic for all $n.$ It follows that $\Vert M_n \Vert$ is bounded from above and bounded away from $0,$ hence it readily follows that for the top-Lyapunov exponent we have $\lambda_1 = 0,$ and we obtain the following result:

	\begin{theorem}
		\label{thm:ei Mnx vs ei Mnw COLUMN STOCH}
		Assume that any of the sets of conditions of Theorems \ref{thm:ei Mnx vs ei Mnw IID}, \ref{thm:ei Mnx vs ei Mnw STATIONARY BOUNDED} or \ref{thm:ei Mnx vs ei Mnw STATIONARY GEN} is satisfied, and in addition $A_n$ is column-stochastic for all $n.$
		Then for any vector of initial values $x \in  \mathbb R^p,$ and any non-negative vector of initial weights $w \in \mathbb R^p_+ $ such that $w \neq 0$ we have for all $i$ a.s.
		\begin{equation*}
			\limsup_{n \rightarrow \infty} {\frac 1 n} \log \left \vert \frac{e_i^\top M_n x}{e_i^\top M_n w}  - \frac{ {\mathbf 1}^\top x}{  {\mathbf 1}^\top w} \right \vert \le \lambda_2 < 0.
		\end{equation*}
	\end{theorem}

Choosing $w= {\mathbf  1},$ Theorem \ref{thm:ei Mnx vs ei Mnw COLUMN STOCH} implies that ratio consensus will take place in the classic sense: for all agents $k$ the values $x_{n}^{k}/w_{n}^{k}$ will converge to the same non-random limit $\bar x = \sum_{i=1}^p x^i_0/p$, with at least the given rate.

\begin{remark}
	It may come as a pleasing surprise that the a.s. rate of convergence for weighted gossip algorithms provided by Theorems \ref{thm:ei Mnx vs ei Mnw COLUMN STOCH} is identical with the a.s. rate of convergence of a class of \emph{linear gossip} algorithms, described in \cite{picci2013almost}, defined via a strictly stationary ergodic edge process. By Theorem 5.2 of \cite{picci2013almost}, with $A_n$ denoting the associated {\emph {doubly stochastic}} matrices, we have for any $x \in \mathbb R^p$ and any $i$
	\begin{equation}
	\label{eq:linear gossip rate}
	\limsup_{n \rightarrow \infty} {\frac 1 n} \log \vert e_i^\top A_n \cdot \ldots \cdot A_1 x -  {\frac {\mathbf 1^T x} {p}} \vert \le \lambda_2 \quad {\rm a.s.}
	\end{equation}
	We note that an extension of this result can be easily derived from the proof of Theorem \ref{thm:ei Mnx vs ei Mnw COLUMN STOCH}: assuming the additional condition that $A_n$ is \emph{doubly stochastic} for all $n$ inequality \eqref{eq:linear gossip rate} holds. Unfortunately the problem of deciding if $\lambda_2 < 0$ is generally not only NP hard, but undecidable \cite{tsitsiklis1997lyapunov}, \cite{blondel2000survey}.

An upper bound for the rate of a.s. exponential convergence of an appropriately sampled process $x^i_{\tau_n}/w^i_{\tau_n},$ generated by the push-sum or weighted gossip algorithms, was derived in  \cite{iutzeler2013analysis} assuming, among others, that $(A_n)$ is i.i.d. and   column-stochastic. 
These upper bounds for the rate, obtained via the analysis of the mean squared error of $A_n \cdots A_1 \cdot (I - {\mathbf 1}{\mathbf 1}^\top/p)$, are given by $\kappa = - {\frac 1 2} \log \rho(R)$, with $\rho(\cdot)$ denoting the spectral radius, where 
	$$
	R = {{\mathbb  E}} [A_1 \otimes  A_1] \cdot ((I - {\mathbf 1}{\mathbf 1}^\top/p) \otimes (I - {\mathbf 1}{\mathbf 1}^\top/p)).
	$$
We should note that that the same \emph{computable} upper bound for the rate of a.s. exponential convergence of the complete process $x^i_{n}/w^i_{n}$ can be readily derived by combining the arguments of  \cite{iutzeler2013analysis} with Lemma \ref{lem:1 over u1n} of the present paper.
\end{remark}

The upper bounds for the rates in the preceding theorems seem to have been unknown prior to this paper. As for the exact rate the best we can claim is the following theorem:
	
\begin{theorem} 
	\label{thm:max i of ratios via barMnx vs barMnw} Assume that any of the sets of conditions of Theorems \ref{thm:ei Mnx vs ei Mnw IID}, \ref{thm:ei Mnx vs ei Mnw STATIONARY BOUNDED} or  \ref{thm:ei Mnx vs ei Mnw STATIONARY GEN} is satisfied. Then for all pairs of non-negative vectors $(x,w) \in \mathbb R^p_+ \times \mathbb R^p_+,$ such that $x,w \neq 0,$ 
	except perhaps for a set of Lebesgue-measure zero, it holds that 
		\begin{equation*}
                  \lim_{n \rightarrow \infty} {\frac 1 n} \log \max_i \left \vert \frac{e_i^\top M_n x}{e_i^\top M_n w}  - \frac{v^{1 \cdot}  {x}}{v^{1 \cdot}  {w}} \right \vert = -(\lambda_1 - \lambda_2) \quad {\rm a.s.}
		\end{equation*}
\end{theorem}

\section{Specification for push-sum with packet loss}
\label{sec:Push Sum}

In this section we summarize the implications of the above stated results for the classic push-sum or weighted gossisp algorithm, allowing packet loss as described in the Introduction, which is in line with the setting of \cite{gerencser2018push}.

\begin{theorem}
	\label{thm:ei Mnx vs ei Mnw PS}
	Let $(A_n)$ be the associated i.i.d.\ sequence of matrices defined under (\ref{eq:def An PS with Loss}). Assume that the directed communication graph $(G,E)$ is strongly connected. Then for any initial values $x \in \mathbb R^p,$ and a non-negative vector of initial weights $w \in \mathbb R^p_+ $ such that $w \neq 0$ ratio consensus takes place and for all $i$-s an explicit upper bound for the a.s. rate of convergence can be given as follows:
	\begin{equation*}
	\label{eq:ei Mnx vs ei Mnw PS}
	\limsup_{n \rightarrow \infty} {\frac 1 n} \log \left \vert \frac{e_i^\top M_n x}{e_i^\top M_n \mathbf 1}  - \frac{v^{1\cdot}x}{v^{1\cdot} \mathbf 1} \right \vert \le -(\lambda_1 - \lambda_2) <0.  
	\end{equation*}
	In the case of no packet loss we have $\lambda_1 = 0$ and $v^{1\cdot} = \mathbf 1^\top.$ 
\end{theorem}

\begin{proof} [Proof of Theorem \ref{thm:ei Mnx vs ei Mnw PS}.] For the first step of the proof we verify the only non-trivial condition of Theorem \ref{thm:barMnx vs barMnw} requiring that $(A_n)$ is sequentially primitive. Since $(A_n)$ is an i.i.d.\ sequence we can resort to Lemma \ref{rem:seq prim vs set prim}. Consider therefore the (finite) range of the random matrices $A_n$ given by (\ref{eq:def An PS with Loss}), denoted by ${\cal A}_{\rm PS}.$ The first of the following two lemmas restates a well known result in the consensus literature, see \cite{benezit2010weighted}, while the second one claims the validity of the key condition $\lambda_1 > \lambda_2$. The proofs will be given in Appendix III:
	
	\begin{lemma}
		\label{lem:cal APS prim}Assume that the directed communication graph $G=(V,E)$ is strongly connected. Then the set ${\cal A}_{\rm PS}$ is primitive.
	\end{lemma}
	
	\begin{lemma}
		\label{lem:gap for PS is positive}
		Let $(A_n)$ be an i.i.d.\ sequence of matrices corresponding to the push-sum algorithm allowing packet loss, defined in (\ref{eq:def An PS with Loss}), satisfying the condition described in the Introduction. Then we have for the spectral gap $\lambda_1-\lambda_2 > 0$.
	\end{lemma}

	To complete the proof of Theorem \ref{thm:ei Mnx vs ei Mnw PS} we apply Theorem \ref{thm:ei Mnx vs ei Mnw IID}, the conditions of which are partially assumed, and partially ensured by the lemmas above. This confirms the general case with possible packet loss. In the case of no packet loss the claim $\lambda_1=0$ and $v^{1 \cdot}= \mathbf 1$ is implied by Theorem \ref{thm:ei Mnx vs ei Mnw COLUMN STOCH}.
\end{proof}

\begin{remark}
Note that the argument used in \cite{atar1997lyapunov} to estimate $\lambda_1 - \lambda_2$ from below can not be used in our case. Namely, \cite{atar1997lyapunov} refers to a result of \cite{peres1992domains}
		$$
		\lambda_1 - \lambda_2 \ge - {\mE} \log \tau(A_1),
		$$
		where $\tau(A_1)$ is the the Birkhoff contraction coefficients of $A_1$ (see below). However, in our case, we have $\tau(A_n)=1$ a.s., hence the lower bound is simply $0.$ 
\end{remark}

By this we end the description of the key points of our work and switch to slightly heavier mathematical details. First we describe the critical steps of the proofs of our main theorems, with some technical details relegated to Appendix IV, and then a mathematical interlude on the spectral gap is added.

\section{Proofs of Theorems \ref{thm:ei Mnx vs ei Mnw IID}, \ref{thm:ei Mnx vs ei Mnw STATIONARY BOUNDED}, \ref{thm:ei Mnx vs ei Mnw STATIONARY GEN}, \ref{thm:ei Mnx vs ei Mnw COLUMN STOCH} and \ref{thm:max i of ratios via barMnx vs barMnw}}
\label{sec:Proofs of Main Theorems}

For the proof of Theorem \ref{thm:ei Mnx vs ei Mnw IID} a natural starting point would be Theorem \ref{thm:barMnx vs barMnw}. However, we will see that nothing is gained compared to a direct proof. On the other hand, the situation is completely different in the case of Theorem \ref{thm:max i of ratios via barMnx vs barMnw}, the proof of which will rely essentially on Theorem \ref{thm:barMnx vs barMnw}.

For the description of the proofs we need the following definition. A stochastic process $\xi_n, n \ge 1$ is called sub-exponential, if for any $\varepsilon >0$ we have for all $n,$ with finitely many exceptions, a.s. $|\xi_n|  ~\le  e^{\varepsilon n}.$ We will use the notation $\xi_n = e^{o(1)n}$. Equivalently, $\xi_n, n \ge 1$ is called sub-exponential if 
$\limsup_{n \rightarrow \infty} {\frac 1 n}  \log |\xi_n| \le 0.$

In view of (\ref{eq:Mn SVD Rank 1}), assuming $\lambda_1 > \lambda_2,$ the matrix product $M_n$ is asymptotically equivalent to the sequence of rank-$1$ matrices $ u_{n}^{\cdot 1} v^{1 \cdot} \sigma^{1}_n,$ a.s. A weak, a priori estimate of a measure of collinearity of the rows of $M_n$ is formalized in Condition \ref{cond:Mnik vs Mnjk subexp}, under which the proofs of Theorems \ref{thm:ei Mnx vs ei Mnw IID} - \ref{thm:max i of ratios via barMnx vs barMnw} will be completed. The validity of Condition \ref{cond:Mnik vs Mnjk subexp} itself will be verified by Lemma \ref{lem:Mnik vs Mnjk subexp} in Appendix IV.

\begin{condition}
\label{cond:Mnik vs Mnjk subexp} Letting $M_n=A_n A_{n-1} \cdots A_1,$ as before, we assume that for any pair of row indices $i,j$, and any column index $k$ it holds that ${M_{n}^{ik}}/{M_{n}^{jk}}$ is sub-exponential. 
\end{condition}

\begin{lemma}
  \label{lem:1 over u1n}
  Under the conditions of Theorem \ref{thm:barMnx vs barMnw}, the additional assumption that $\lambda_1 > \lambda_2,$ and Condition \ref{cond:Mnik vs Mnjk subexp}, it holds that $1/u^{i 1}_{n}$ is sub-exponential a.s. for all $i.$
\end{lemma}

\begin{proof} [Proof of Lemma \ref{lem:1 over u1n}]  
  Recall that according to (\ref{eq:Mn SVD Rank 1}) we have a.s. $M_n = u_{n}^{\cdot 1} v^{1 \cdot} \sigma^{1}_n + O(e^{(\lambda_2 + o(1))n}).$
  Take an arbitrary pair of row indices $j,i,$ and compare the rows $M_n^{j \cdot}$ and $M_n^{i \cdot}.$ Choosing a column index $k$ such that $v^{1 k} > 0$ we consider 
  \begin{equation}
    \label{eq:Mn ik vs jk}
    \frac {M_{n}^{jk}} {M_{n}^{ik}} = \frac {u^{j 1}_{n}  v^{1 k}  \sigma^{1}_n +  O(e^{(\lambda_2 + o(1))n})} 
    {u^{i 1}_{n}   v^{1 k} \sigma^{1}_n +  O(e^{(\lambda_2 + o(1))n})}.
  \end{equation}
  Taking into account $ v^{1 k} >0,$ we would have for any $j,i$
  \begin{equation}
    \label{eq:Mn ik vs jk FNL}
    ~\frac {M_{n}^{jk}} {M_{n}^{ik}} = 
    \frac {u^{j 1}_n   +  O(e^{(-\lambda_1 + \lambda_2 + o(1))n})} 
    {u^{i 1}_n   +  O(e^{(-\lambda_1 +\lambda_2 + o(1))n})}.
  \end{equation}
  
From this it follows that $1/{u^{i 1}_{n}}$ is sub-exponential as stated. Indeed, assume that this not the case, then for some small $\varepsilon >0$ we have $1/ {u^{i 1}_{n}} \ge e^{\varepsilon n}$ for an infinite subsequence, say $n=n_r,$ consequently ${u^{i 1}_{n}} \le e^{-\varepsilon n}$ for $n=n_r.$ Select $j$ so that for some infinite subsequence of $(n_r)$, which we identify with $(n_r)$, we have ${u^{j 1}_{n_r}} \ge 1/\sqrt p.$ 
The indirect assumption and the choice of $j$ would then imply ${M_{n}^{jk}}/{M_{n}^{ik}} \ge C e^{\varepsilon n} $ with some $C > 0$ infinitely many times a.s., which is a contradiction to Condition \ref{cond:Mnik vs Mnjk subexp}. 
\end{proof}

	\begin{lemma}
		\label{lem:v1cdot positive}
		Under the conditions of Theorem \ref{thm:barMnx vs barMnw}, with the additional assumption that $\lambda_1 > \lambda_2,$ and Condition \ref{cond:Mnik vs Mnjk subexp}, it holds that $v^{1 i} > 0$ for all $i=1,\ldots,p.$
	\end{lemma}

\begin{proof} [Proof of Lemma \ref{lem:v1cdot positive}] Consider the matrix process ${\bar A}_n = A_{-n}^\top.$ First we show that the Lyapunov exponents for the processes $({\bar A}_n)$ and $({A}_n)$ are identical, ${\bar \lambda}_k=\lambda_k$ for all $k=1,\ldots,p.$ Define for any pair of integers $n > m $ the products ${ M}_{n,m} ={ A}_n  { A}_{n-1} \cdots { A}_{m}$ and ${\bar M}_{n,m} ={\bar A}_n  {\bar A}_{n-1} \cdots {\bar A}_{m}$. Then we have 
\begin{align}
{M}_{n,m}^\top &= \left (A_n  A_{n-1} \cdots A_{m}\right)^\top 
=    A_{m}^\top \cdots A_{n-1}^\top  A_n^\top  \nonumber \\ 
&=  {\bar A}_{-m} \cdots {\bar A}_{-n+1}  {\bar A}_{-n} = {\bar M}_{-m,-n}. \nonumber
\end{align}

Let a singular value decomposition (SVD) of ${M}_{n,m}$ be  
\begin{equation*}
{M}_{n,m} = U_{n,m} \Sigma_{n,m} V_{n,m}.  
\end{equation*}
Then an SVD for ${\bar M}_{-m,-n}$ is obtained as follows: 
\begin{equation*}
{\bar M}_{-m,-n}=V_{n,m}^\top \Sigma_{n,m} U_{n,m}^\top  =:{\bar U}_{-m,-n} {\bar \Sigma}_{-m,-n} {\bar V}_{-m,-n}.
\end{equation*}
with the notations
\begin{align}
{\bar U}_{-m,-n} &= V_{n,m}^\top, 
\label{eq:SVD vs bar SVD 1} \\ 
{\bar \Sigma}_{-m,-n} &= \Sigma_{n,m}, \label{eq:SVD vs bar SVD 2} \\ 
{\bar V}_{-m,-n} &= U_{n,m}^\top. \label{eq:SVD vs bar SVD 3}
\end{align}

To prove ${\bar \lambda}_1=\lambda_1$ note that \eqref{eq:SVD vs bar SVD 1}-\eqref{eq:SVD vs bar SVD 3} implies:
\begin{align}
{\bar \lambda}_1 = \lim_{-m \rightarrow \infty} &{\frac 1 {n-m+1}} \log {\bar \sigma}_{-m,-n}^{1} \nonumber \\
=  
\lim_{-m \rightarrow \infty} &{\frac 1 {n-m+1}} \log {\sigma}_{n,m}^{1}  \nonumber
\end{align}
w.p.1, and hence also in distribution. But ${\sigma}_{n,m}^{1}$ and ${\sigma}_{n-m+1,1}^{1}$ have the same distribution, and for the latter we have 
\begin{equation*}
{\lambda}_1 = \lim_{-m \rightarrow \infty} {\frac 1 {n-m+1}} \log { \sigma}_{n-m+1,1}^{1} 
\end{equation*}
w.p.1, and hence also in distribution. Thus the distribution of ${\bar \lambda}_1$ and $\lambda_1$ agree implying ${\bar \lambda}_1=\lambda_1.$

Applying the same argument to the $k$-th exterior product sequences formed by $A_n \wedge \ldots \wedge A_n$ and ${\bar A}_n \wedge \ldots \wedge {\bar A}_n$ we conclude that ${\bar \lambda}_1 + \ldots + {\bar \lambda}_k ={\lambda}_1 + \ldots + {\lambda}_k$ for all $k$ implying the claim.

Next, consider the matrices $ {V}_{n,m}$ with $m$ fixed and $n$ tending to $\infty.$ The first rows of $ {V}_{n,m}$ denoted by $ {v}_{n,m}^{1 \cdot} $ converge a.s. to a limit, say $ {v}_{m}^{1 \cdot}$ with exponential rate by Lemma 5 of \cite{raghunathan-oseledec}, the error being $O(e^{(-\lambda_1 + \lambda_2 + o(1))(n-m)}).$ This implies, that the first columns of ${\bar U}_{-m,-n}$, denoted by ${\bar u}_{-m,-n}^{\cdot 1}$ also converge to a limit ${\bar u}_{-m}^{\cdot 1} = {v}_{m}^{1 \cdot \top}$ a.s. with the same exponential rate when $n$ tends to $\infty.$

Take $m=1$ and assume in contrary to the statement of the lemma that $v^{1 i}=  {v}_{1}^{1 i} = 0$ for some $i.$ Then ${\bar u}_{-1}^{i 1} = 0,$ and thus ${\bar u}_{-1,-n}^{i 1}$ is exponentially small a.s. when $n$ tends to $\infty:$ writing $\xi_n := {\bar u}_{-1,-n}^{i 1}$ we have for any $0 < \mu < \lambda_1 - \lambda_2$ with some $C(\omega)>0$ the inequality $\xi_n \le C(\omega) e^{-\mu n}.$ This implies for the distribution of $\xi_n$ that for any $\mu' < \mu  < \lambda_1 - \lambda_2$ 
\begin{align}
		\label{eq:xin le exp minus mu tends to 1}
		 & P(\xi_n \le e^{-\mu' n}) \ge P( C(\omega) e^{-\mu n} \le   e^{-\mu' n}) \nonumber \\ 
		= & P( C(\omega) \le   e^{(\mu - \mu') n}) \rightarrow 1, \quad {\rm as} \quad n \rightarrow \infty. 
\end{align}

	On the other hand, shifting the time indices in ${\bar u}_{-1,-n}^{i 1}$ by $n+1$ we get the random variables $\xi'_n := {\bar u}_{n,1}^{i 1}$ having the same distribution as $\xi_n.$ Applying Lemma \ref{lem:1 over u1n} to the process $({\bar A}_n),$ where the conditions are easily verified, we get that $1/\xi'_n$ is sub-exponential. Thus for any $\varepsilon >0$ we have $1/\xi'_n \le C'(\omega) e^{\varepsilon n}$, with some $ C'(\omega) >0.$ Following the argument given above we get for the distribution of $1/\xi'_n$ that for any $\varepsilon' > \varepsilon >0$ it holds that $P(1/\xi'_n \le e^{\varepsilon' n}) \rightarrow 1$ as $n \rightarrow \infty,$ implying $P( e^{-\varepsilon' n} \le \xi'_n ) \rightarrow 1,$ which in turn yields
		\begin{equation}
		\label{eq:xin prime le exp tends to 0}
		P(\xi'_n < e^{-\varepsilon' n}) \rightarrow 0, \quad {\rm as} \quad n \rightarrow \infty. 
		\end{equation}
		Choosing $0< \varepsilon < \varepsilon' < \mu',$ and recalling that $\xi'_n$ and $\xi_n$ have the same distribution, we get a contradiction with \eqref{eq:xin le exp minus mu tends to 1}, and thus the proof is complete. 
\end{proof}

\begin{proof}[Proofs of Theorems \ref{thm:ei Mnx vs ei Mnw IID}, \ref{thm:ei Mnx vs ei Mnw STATIONARY BOUNDED} and \ref{thm:ei Mnx vs ei Mnw STATIONARY GEN}:] Assuming the validity of Condition \ref{cond:Mnik vs Mnjk subexp}, to be established separately under each set of conditions of Theorems \ref{thm:ei Mnx vs ei Mnw IID}, \ref{thm:ei Mnx vs ei Mnw STATIONARY BOUNDED}, \ref{thm:ei Mnx vs ei Mnw STATIONARY GEN}, the proof of the quoted three theorems are identical:

Recall that we have by \eqref{eq:Mn SVD Rank 1} $M_n = u_{n}^{\cdot 1} v^{1 \cdot} \sigma^{1}_n + O(e^{(\lambda_2 + o(1))n}),$ hence 
\begin{equation}
  \label{eq:ek Mnx vs ek Mnww 0}
  \frac{e_i^\top M_{n} x}{e_i^\top M_{n} w} = 
  \frac {e_i^\top u^{\cdot 1}_{n}  v^{1 \cdot} x  \sigma^{1}_n +  O(e^{(\lambda_2 + o(1))n})} 
  {e_i^\top u^{\cdot 1}_{n}   v^{1 \cdot} w \sigma^{1}_n +  O(e^{(\lambda_2 + o(1))n})}.
\end{equation}
Dividing both the numerator and the denominator by $\sigma^{1}_n,$ we get 
\begin{equation}
  \label{eq:ek Mnx vs ek Mnw fnl}
  \frac{e_i^\top M_{n} x}{e_i^\top M_{n} w} = 
  \frac {e_i^\top u^{\cdot 1}_{n} \cdot v^{1 \cdot} x   +  O(e^{(-\lambda_1 + \lambda_2 + o(1))n})} 
  {e_i^\top u^{\cdot 1}_{n} \cdot  v^{1 \cdot} w +  O(e^{(-\lambda_1 + \lambda_2 + o(1))n})}.
\end{equation}
Note that $v^{1 \cdot} > 0$ by Lemma \ref{lem:v1cdot positive}, and thus $w \ge 0, ~ w \neq 0$ imply $v^{1 \cdot} w>0.$ Divide both the numerator and the denominator by $v^{1 \cdot} w$ and also by $e_i^\top u^{\cdot 1}_{n}.$ The proof is then completed by noting that $1/e_i^\top u^{\cdot 1}_{n} = 1/u^{i 1}_{n}$ is sub-exponential for all $i,$ as stated in Lemma \ref{lem:1 over u1n}.
\end{proof}

\begin{proof} [Proof of Theorem \ref{thm:ei Mnx vs ei Mnw COLUMN STOCH}] First note that $M_n$ is column-stochastic for all $n,$ hence $\Vert  M_n \Vert$ is bounded from above and bounded away from zero. It follows that $\lambda_1 = 0.$ To complete the proof it is sufficient to show that $v^{1 \cdot}$ is proportional to ${\mathbf 1}^\top,$ (implying that $v^{1 \cdot} = {\mathbf 1}^\top / \sqrt p.$) Writing
	\begin{equation}
	{\mathbf 1}^\top = {\mathbf 1}^\top M_n = {\mathbf 1}^\top 
	u^{\cdot 1}_{n}  v^{1 \cdot}  \sigma^{1}_{n} +  O(e^{(\lambda_2 + o(1))n}) \quad {\rm a.s.},
	\end{equation}
	and noting that ${\mathbf 1}^\top u^{\cdot 1}_{n}$ and $\sigma^{1}_{n} = \Vert  M_n \Vert$ are bounded and bounded away from $0$, after dividing by these we get
	\begin{equation}
	c_n {\mathbf 1}^\top = v^{1 \cdot}  +  O(e^{(\lambda_2 + o(1))n}) \quad {\rm a.s.,}
	\end{equation}
	with some possibly random scalar $c_n. $ Letting $n \rightarrow \infty,$ and taking into account $\lambda_2 <0,$ the r.h.s. will converge to $v^{1 \cdot},$ and thus the l.h.s. will also converge, implying that $c_n$ converges to some $c,$ yielding $c {\mathbf 1}^\top = v^{1 \cdot},$ as claimed.  
\end{proof}

\begin{proof} [Proof of Theorem \ref{thm:max i of ratios via barMnx vs barMnw}] 
Note that the a.s. inequality 
	\begin{equation}
		\label{eq:ek Mnx vs ek Mnw limsup}
		\limsup_{n \rightarrow \infty} {\frac 1 n} \max_i \log \left \vert \frac{e_i^\top M_n x}{e_i^\top M_n w}  - \frac{v^{1 \cdot}  {x}}{v^{1 \cdot}  {w}} \right \vert \le -(\lambda_1 - \lambda_2) 
	\end{equation}
follows directly from Theorem \ref{thm:ei Mnx vs ei Mnw IID}. 	
For the proof that the inequality is actually an equality we will rely on Theorem \ref{thm:barMnx vs barMnw}. First note that, in addition to $w >0$ we may assume $x >0,$ since the set of pairs $(x, w) \in \mathbb R^p \times  \mathbb R^p,$ having a $0$ component in $x$ have zero Lebesgue measure. Now, note that for any pairs or probability vectors $(\bar x, \bar w)$ we have 
$$
\max_i \vert \bar{x}_i- \bar{w}_i \vert \le    \Vert \bar{x}- \bar{w} \Vert_{\rm TV} \le p \max_i \vert \bar{x}_i- \bar{w}_i \vert.
$$
Therefore Theorem \ref{thm:barMnx vs barMnw} can be restated as follows: for all pairs $(x,w) \in  \mathbb R_+^p \times  \mathbb R_+^p, ~x,w \neq 0,$ except for a set of Lebesgue-measure zero, it holds a.s. that
\begin{align}
\label{eq:max i barMnx vs barMnw }
& \lim_{n \rightarrow \infty} \frac{1}{n} \log \max_i \vert \bar{x}_{n}^{i}- \bar{w}_{n}^{i} \vert \nonumber \\
 =
&\lim_{n \rightarrow \infty} \frac{1}{n} \log \max_i \left \vert \frac {{x}_{n}^{i}} {\mathbf 1^\top x_n}- \frac {{w}_{n}^{i}} {\mathbf 1^\top w_n} \right \vert = - (\lambda_1 -
\lambda_2). 
\end{align}
We may relate this equality to a ratio consensus problem by rewriting the middle term as 
\begin{align}
\label{eq:max i barMnx vs barMnw V3}
& \lim_{n \rightarrow \infty} \frac{1}{n} \log \max_i \left \vert \frac {{x}_{n}^{i}} {{w}_{n}^{i} }- \frac  {\mathbf 1^\top {x}_n} {\mathbf 1^\top   {w}_n} \right \vert \cdot \frac   {{w}_{n}^{i} }   {{\mathbf 1}^\top {x}_n} \nonumber \\
 = & \lim_{n \rightarrow \infty} \frac{1}{n} \max_i \left( \log \left \vert \frac {{x}_{n}^{i}} {{w}_{n}^{i} }- \frac  {{\mathbf 1}^\top x_n} {{\mathbf 1}^\top w_n} \right \vert +  \log \frac   {{w}_{n}^{i} }   {{\mathbf 1}^\top x_n} \right).
\end{align}
Now, if $a_i, b_i$ are real numbers, then $\max_i (a_i + b_i) \le \max_i a_i + \max_i  b_i.$ 
Apply this inequality to the r.h.s. of 
(\ref{eq:max i barMnx vs barMnw V3}) and take into account \eqref{eq:max i barMnx vs barMnw } to get that $-(\lambda_1 - \lambda_2)$ is bounded from above by 
\begin{equation*}
\liminf_n \frac{1}{n} \left( \max_i \log \left \vert \frac {{x}_{n}^{i}} {{w}_{n}^{i} }- \frac  {{\mathbf 1}^\top x_n} {{\mathbf 1}^\top w_n} \right \vert +  \max_i \log \frac   {{w}_{n}^{i} }   {{\mathbf 1}^\top x_n} \right).
\end{equation*}
Furthermore, if $\alpha_n, \beta_n, n \ge 1,$ are real numbers and $\gamma_n = \alpha_n + \beta_n$ then $\liminf_n \gamma_n \le \liminf_n \alpha_n + \limsup_{n \rightarrow \infty} \beta_n.$ 
(For the verification recall that $\gamma_n \le \alpha_n + \sup_{n\ge 1} \beta_n =: \alpha_n + B$, yielding $\inf_{n \ge m} \gamma_n \le \inf_{n \ge m} (\alpha_n + B) = \inf_{n \ge m} \alpha_n + B$.) Also note that ${{ w}_{n}^{i} } \le \mathbf 1^\top {{w}_{n}}$ implies $\max_i \log ({{w}_{n}^{i}} / {{\mathbf 1}^\top x_n})\le \log (\mathbf 1^\top w_n / {{\mathbf 1}^\top x_n})$. Thus we get 
\begin{align}
\label{eq:max i barMnx vs barMnw V5}
-(\lambda_1 - \lambda_2) \le 
&\liminf_n \frac{1}{n}  \max_i \log \left \vert \frac {{x}_{n}^{i}} {{w}_{n}^{i} }- \frac  {{\mathbf 1}^\top x_n} {{\mathbf 1}^\top w_n} \right \vert \nonumber \\ 
+ & \limsup_{n \rightarrow \infty} \frac{1}{n} \log \frac   {{\mathbf 1}^\top {w}_{n} }   {{\mathbf 1}^\top x_n} .
\end{align}
Now, by Corollary \ref{cor:q Mnx vs q Mnw} $\mathbf 1^\top w_n / {{\mathbf 1}^\top x_n}$ has a finite, non-zero limit w.p.1, hence
$$
\limsup_{n \rightarrow \infty} \frac{1}{n} \log \frac {{\mathbf 1}^\top w_{n}}  {{\mathbf 1}^\top x_n} = \lim_{n \rightarrow \infty} \frac{1}{n} \log \frac {\mathbf 1^\top w_n}  {{\mathbf 1}^\top x_n} = 0.
$$
Hence we conclude that 
\begin{equation}
\label{eq:gap vs liminf max}
- (\lambda_1 - \lambda_2) \le \liminf_n \frac{1}{n} \max_i \log \left \vert \frac {{x}_{n}^{i}} {{w}_{n}^{i} }- \frac  {{\mathbf 1}^\top x_n} {{\mathbf 1}^\top w_n} \right \vert,  
\end{equation}
and combining this with \eqref{eq:ek Mnx vs ek Mnw limsup} we can write equality and $\lim$ in place of $\liminf$ on the right hand side: 
\begin{equation}
\label{eq:gap vs lim max}
- (\lambda_1 - \lambda_2) = \lim_{n \rightarrow \infty} \frac{1}{n} \max_i \log \left \vert \frac {{x}_{n}^{i}} {{w}_{n}^{i} }- \frac  {{\mathbf 1}^\top x_n} {{\mathbf 1}^\top w_n} \right \vert. 
\end{equation}

Now, in view of Corollary \ref{cor:q Mnx vs q Mnw} we have 
\begin{equation}
\min_i \frac {{x}_{n}^{i}} {{w}_{n}^{i} } \le \frac  {{\mathbf 1}^\top x_n} {{\mathbf 1}^\top w_n} \le \max_i \frac {{x}_{n}^{i}} {{w}_{n}^{i} }.
\end{equation}
On the other hand, the trivial inequalities  
\begin{align}
\label{eq: min vs center vs max}
\frac 1 2  & \left \vert \max_i \frac {{x}_{n}^{i}} {{w}_{n}^{i}}- \min_i \frac {{x}_{n}^{i}} {{w}_{n}^{i}} \right \vert \le \max_i \left \vert \frac {{x}_{n}^{i}} {{w}_{n}^{i} }- \frac  {{\mathbf 1}^\top x_n} {{\mathbf 1}^\top w_n} \right \vert  \nonumber \\  \le  & \left \vert \max_i \frac {{x}_{n}^{i}} {{w}_{n}^{i}}- \min_i \frac {{x}_{n}^{i}} {{w}_{n}^{i}} \right \vert 
\end{align}
combined with \eqref{eq:gap vs lim max} yield
\begin{equation}
\label{eq:gap vs lim max-min}
- (\lambda_1 - \lambda_2) = \lim_{n \rightarrow \infty} \frac{1}{n} \log\left \vert \max_i  \frac {{x}_{n}^{i}} {{w}_{n}^{i} }-\min_i  \frac {{x}_{n}^{i}} {{w}_{n}^{i} } \right \vert \quad {\rm a.s.}  
\end{equation}
except for a set of initial $(x,w)$-s of Lebesgue measure zero. Considering \eqref{eq: min vs center vs max} and replacing $ {{\mathbf 1}^\top x_n} /{{\mathbf 1}^\top w_n}$ by an arbitrary sequence of intermediate values $v_n$ such that 
$$
\min_i  \frac {{x}_{n}^{i}} {{w}_{n}^{i} } \le v_n \le \max_i  \frac {{x}_{n}^{i}} {{w}_{n}^{i} } 
$$
we get by the same logic 
\begin{equation}
\label{eq:gap vs liminf max-min}
- (\lambda_1 - \lambda_2) = \lim_{n \rightarrow \infty} \frac{1}{n} \max_i \log  \left \vert  \frac {{x}_{n}^{i}} {{w}_{n}^{i} }- v_n \right \vert.  
\end{equation}
Taking $v_n=  {v^{1 \cdot}  {x}} /{v^{1 \cdot}  {w}}$ for all $n,$ in view of Lemma \ref{lem:min ratio and max ratio monotone}, we get the claim.
\end{proof}

\begin{remark}In the special case when $M_n$ is column-stochastic, we have ${\mathbf 1}^\top  {x}_{n} = {\mathbf 1}^\top  M_n x = {\mathbf 1}^\top  {x}$, and similarly ${\mathbf 1}^\top w_n  = {\mathbf 1}^\top w$ for all $n.$ Furthermore, by Theorem \ref{thm:ei Mnx vs ei Mnw COLUMN STOCH} we have $v^{1 \cdot} = \mathbf 1^\top.$ Thus, in this special case \eqref{eq:gap vs lim max} immediately implies the claim without any further deliberations. 
\end{remark}

\section{A representation of the spectral gap  $\lambda_1 - \lambda_2$}
\label{sec:Gap vs Birkhoff}

As we have seen, the spectral gap $\lambda_1 - \lambda_2$ plays a key role in characterizing the stability of normalized products and the convergence rate of the ratio consensus method. In this section we present a set of simple results providing computable lower bounds and alternative representations for the spectral gap under the conditions of Theorems \ref{thm:barMnx vs barMnw}, \ref{thm:ei Mnx vs ei Mnw IID}, \ref{thm:ei Mnx vs ei Mnw STATIONARY BOUNDED} or \ref{thm:ei Mnx vs ei Mnw STATIONARY GEN}.

A lower bound for the spectral gap was established in \cite{peres1992domains}, Proposition 5, under the condition that $A_1$ is strictly positive with positive probability. In fact this result is a simple corollary of Theorem \ref{thm:barMnx vs barMnw} relying on its less restrictive conditions. For the formal statement we introduce the following definitions and notations.

\begin{definition}
\label{def:Hilbert distance} Let $x,y \in \mathbb R^p_+$ be strictly positive vectors, $x, y >0.$ Then their Hilbert-distance is defined as 
\begin{equation}
h(x,y)  := \log \max_{k,l} \left( \frac {x_k} {y_k} \big /  \frac {x_l} {y_l} \right).
\end{equation}
\end{definition}

The Hilbert-distance satisfies the properties of a metric within the set of strictly positive vectors in $\mathbb R^p,$ except that $h(x,y)=0$ if and only if $y = c x$ with some $c >0.$ The operator norm of a non-negative allowable matrix $A$ corresponding to the Hilbert-distance is called the Birkhoff contraction coefficient of $A.$ More exactly we set

\begin{definition}
\label{def:Birkhoff coeff} The Birkhoff contraction coefficient of a non-negative allowable matrix $A$ is defined as 
\begin{equation*}
\tau(A) :=\sup \left \{\frac {h(Ax,Ay)} {h(x,y)} ~\bigg|~ x,y \in \mathbb R^p_+,  ~{h(x,y)} \neq 0 \right \}.  
\end{equation*}	
\end{definition} 	

Note that $x,y > 0$ and the assumption that $A$ is allowable imply that $Ax, Ay > 0,$ and thus $h(Ax,Ay)$ is well-- defined. Obviously, $\tau(A)$ is sub-multiplicative, i.e. $\tau(AB) \le \tau(A) \cdot \tau(B),$ and it is easy to see that $\tau(A) \le 1.$

A beautiful theorem due to Birkhoff yields an explicit expression of $\tau(A)$ in terms of the elements of $A,$ which we present for allowable matrices. Define an intermediary quantity $\varphi (A)$ as follows. Let $\varphi (A)=0$ if $A$ has any $0$ element. Otherwise, we set 
\begin{equation}
\label{eq:def varphi A}
\varphi (A) :=  \log \max_{i,j,k,l}  \left( \frac {A^{ik}} {A^{jk}} \right) / \left(  \frac {A^{il}} {A^{jl}}  \right) = \max_{i,j} 
h(A^{i \cdot}, A^{j \cdot}).
\end{equation}	
By Birkhoff's theorem (Theorem 3.12 of \cite{seneta2006non} or \cite{cavazos2003alternative}) 
\begin{equation}
\label{eq:Birkhoff vs tanh}
\tau(A) = \tanh ~\left(\frac {\varphi(A)} 4 \right)= \frac {e^{\varphi(A)/4} - e^{\varphi(A)/4}} {e^{\varphi(A)/4} + e^{\varphi(A)/4}}. 
\end{equation}

\begin{theorem}
\label{thm:gap vs Birkhoff A1}
Let $(A_n)$ be a strictly stationary, ergodic stochastic process of $p \times p$ matrices satisfying the conditions of Theorem \ref{thm:barMnx vs barMnw}. 
Then 
$$
\lambda_1-\lambda_2 \ge  - \mE \log\tau(A_1).
$$
\end{theorem}

\begin{proof} [Proof of Theorem \ref{thm:gap vs Birkhoff A1}.]
Since $A_m$ is allowable for all $m$ and $x,w$ are strictly positive, the Hilbert-distances of $x_n= A_n A_{n-1}\cdots A_1 x$ and $w_n= A_n A_{n-1}\cdots A_1 w$ are well-defined, and we have 
\begin{align}
\label{eq:Hilbert xn wn le}
h(x_n, w_n) & =  h ( A_n A_{n-1}\cdots A_1 x, ~A_n A_{n-1}\cdots A_1 w ) \nonumber \\
&\le \prod_{k=1}^n \tau(A_k)\cdot h(x,w). 
\end{align}
Therefore we get:
\begin{align}
\label{eq:limsup h Mnx Mnw}
&\limsup_{n \rightarrow \infty} \frac{1}{n} \log h(x_n, w_n) \le \limsup_{n \rightarrow \infty} {\frac 1 n} \sum_{k=1}^n \log \tau(A_k) \nonumber \\ 
= & \lim_{n \rightarrow \infty} {\frac 1 n} \sum_{k=1}^n \log \tau(A_k) = \mE \log\tau(A_1) \quad {\rm a.s.},
\end{align}
where the last two equalities follow from the ergodic theorem. Note that we can handle also the case when $ \mE \log\tau(A_1) =- \infty$ since $\log\tau(A_1)$ is bounded from above by $0.$ Now, the left hand side can be bounded from below via the total variation $||\bar x_n - \bar w_n||_{\rm TV}$ using the following elementary lemma:

\begin{lemma}
\label{lem:TV vs Hilbert dist}
Let $\xi,\eta$ be two strictly positive probability vectors in $\mathbb R^p.$ Then for their total variation distance we have 
$$
\|\xi - \eta\|_{\rm TV} \le \frac 1 2 ~(e^{h(\xi, \eta)} - 1).
$$
\end{lemma}

\begin{proof} [Proof of Lemma \ref{lem:TV vs Hilbert dist}]
Let us write briefly $h=h(\xi, \eta).$ First note that for any $k,l$ we have
$$
\frac{\xi_k}{\eta_k} / \frac{\xi_l}{\eta_l} \le e^h.
$$
Define $R = \max_k \frac{\xi_k}{\eta_k}$, $r = \min_l
\frac{\xi_l}{\eta_l}.$ Since $\xi,\eta$ are
probability vectors, we have $R \ge 1 \ge r,$ and thus from the above inequality we get $e^{-h} \le r  \le R \le e^h$. Taking a $k$ such that $\xi_k \ge \eta_k$ we have 
$$
|\xi_k - \eta_k| = \xi_k - \eta_k = \left (\frac{\xi_k}{\eta_k} - 1 \right) \eta_k \le (e^h - 1) \eta_k.
$$
On the other hand, for $\xi_k \le \eta_k$ we get 
$$
|\xi_k - \eta_k| =  \eta_k - \xi_k  = \left( 1 - \frac{\xi_k}{\eta_k} \right) \eta_k  \le  (1 - e^{-h})\eta_k \le (e^h - 1) \eta_k.
$$
Summation over $k$ gives the claim. 
\end{proof}

To complete the proof of Theorem \ref{thm:gap vs Birkhoff A1} we note that due to the lemma above we can bound $h=h(\xi, \eta)$ from below for small $h,$ say for $0 \le h \le 1/2$ we get $\|\xi - \eta\|_{\rm TV} \le h.$ Taking into account that the Hilbert-distance is invariant w.r.t. scaling its arguments we have $h(x_n, w_n) = h(\bar x_n, \bar w_n),$ and this is exponentially small by Theorem \ref{thm:barMnx vs barMnw}, thus we can use $\|\xi - \eta\|_{\rm TV} \le h$ in (\ref{eq:limsup h Mnx Mnw}) to get
\begin{align}
& \limsup_{n \rightarrow \infty} {\frac 1 n} \log \|\bar x_n - \bar w_n\|_{\rm TV} \le 
\limsup_{n \rightarrow \infty} {\frac 1 n} \log h(\bar x_n, \bar w_n) \nonumber \\
= & \limsup_{n \rightarrow \infty} {\frac 1 n} \log h(x_n, w_n)
\le \mE \log\tau(A_1) \quad {\rm a.s.} 
\end{align}
But we know by Theorem \ref{thm:barMnx vs barMnw} that for almost all pairs $(x,w), x >0, w >0,$ the left side is equal to $-(\lambda_1-\lambda_2)$ a.s., even with $\lim$ instead of $\limsup.$ From here after rearrangement we get the claim. 
\end{proof}

Note that the above result is directly not applicable for the analysis of the push-sum algorithm allowing packet loss, since all off-diagonal elements of $A_1,$ except at most one, is $0$ and hence $\tau(A_1) \equiv 1$ for all $\omega$. A set of alternative lower bounds can be obtained by segmenting the product $A_n \cdots A_1$ into the product of blocks of fixed length, say $m \ge 1.$ Let $A_n(\omega)= A_1(T^n \omega),$ where $T$ is a measure-preserving ergodic transformation of $\Omega.$ Theorem \ref{thm:gap vs Birkhoff A1} has the following extension:

\begin{theorem}
  \label{thm:gap vs Birkhoff Am A1}
  Let $(A_n)$ be a strictly stationary, ergodic stochastic process of $p \times p$ matrices satisfying the conditions of Theorem \ref{thm:barMnx vs barMnw}. Then for all integers integers $m \ge 1$ we have 
  \begin{equation}
    \label{eq:gap vs lim Birkhoff}
    \lambda_1-\lambda_2 \ge   - \frac{1}{m} \mE \log\tau(M_m).
  \end{equation}
\end{theorem}

\begin{proof} [Proof of Theorem \ref{thm:gap vs Birkhoff Am A1}.]
		Let $m \ge 1,$ and define $B_n = A_{nm} \cdot A_{nm -1}\cdot \cdots A_{(n-1)m +1}.$ Obviously, $B_{n+1}(\omega)=B_{n}(T^m \omega),$ thus $(B_n)$ is a strictly stationary process. Now, in analogy with \eqref{eq:Hilbert xn wn le} we have 
		\begin{align}  
		\label{eq:Hilbert xnm wnm le}
		 h(x_{nm}, w_{nm}) & =  h ( B_n B_{n-1}\cdots B_1 x, ~B_n B_{n-1}\cdots B_1 w ) \nonumber \\ 
		& \le \prod_{k=1}^n
		\tau(B_k)\cdot h(x,w). 
		\end{align}
		Therefore we get:
		\begin{align} 
		\label{eq:limsup h Mnmx Mnmw}
		\limsup_{n \rightarrow \infty}  & \frac{1}{nm} \log h(x_{nm}, w_{nm}) \le \limsup_{n \rightarrow \infty} {\frac 1 {nm}} \sum_{k=1}^n \log \tau(B_k) \nonumber \\ =  & {\frac 1 m} \lim_{n \rightarrow \infty} {\frac 1 n} \sum_{k=1}^n \log \tau(B_k) \quad {\rm w.p.1},
		\end{align} 
		where the last equality follows from the ergodic theorem. Here the left hand side is bounded from below by $-(\lambda_1-\lambda_2)$ w.p.1. as seen above. 
		Applying the ergodic theorem once again the right hand side converges to ${\frac 1 m} \mE ~[\log \tau(B_1) ~|~ {\cal F}_{T^m} ],$ where ${\cal F}_{T^m}$ denotes the $\sigma$-algebra of invariant sets w.r.t. $T^m.$ Thus we get the almost sure upper bound for $	-(\lambda_1-\lambda_2)$:
                \begin{equation*}
                  {\frac 1 m} \lim_{n \rightarrow \infty} {\frac 1 n} \sum_{k=1}^n \log \tau(B_k) = {\frac 1 m} \mE ~[\log \tau(B_1) ~|~ {\cal F}_{T^m} ].
		\end{equation*}
                Taking expectation of both sides we get the claim.  
\end{proof}

	Now, it is easy to see that the sequence $\mE \log\tau(M_m)$ is sub-additive (for any ergodic $T$), therefore $\mE \log\tau(M_m)/m$ has a limit (the value of which may be $-\infty$). In addition   
	$$
	\lim_{m \rightarrow \infty} \frac{1}{m} \mE \log\tau(M_m) = \inf_m \frac{1}{m}  \mE  \log\tau(M_m).  
	$$
	Thus we get the following corollary:

\begin{corollary}
	\label{cor:gap vs lim Birkhoff}
	Let $(A_n)$ be a strictly stationary, ergodic stochastic process of $p \times p$ matrices satisfying the conditions of Theorem \ref{thm:barMnx vs barMnw}. Then $	\lambda_1-\lambda_2$ is bounded from below by 
	\begin{equation}
	\label{eq:gap vs lim Birkhoff}
	\lim_{m \rightarrow \infty} - \frac{1}{m} \mE \log\tau(M_m) =  \sup_m - \frac{1}{m}  \mE  \log\tau(M_m).
	\end{equation}
\end{corollary}

A nice application of Corollary \ref{cor:gap vs lim Birkhoff}, providing a lower bound for the spectral gap, is the following:

\begin{theorem}
	\label{thm:gap for seq primitive proc}
	Let $(A_n)$ be a strictly stationary, ergodic stochastic process of $p \times p$ matrices satisfying the conditions of Theorem \ref{thm:barMnx vs barMnw}. Then we have $\lambda_1-\lambda_2 >0.$	
\end{theorem}

\begin{proof}Since $(A_n)$ is sequentially primitive, there exists a finite $m$ such that $P (M_m > 0) > 0.$ But then $P (\tau(M_m) < 1) > 0,$ and hence $ - \mE  \log\tau(M_m) >0.$ The claim now follows from the second part of Corollary \ref{cor:gap vs lim Birkhoff}.
\end{proof}

A natural question that arises at this point if we can drop the expectation on the right hand sides of (\ref{eq:gap vs lim Birkhoff}). We show that in fact this can be done using Kingmans's sub-additive ergodic theorem, see \cite{hammersley1965first, kingman1968ergodic, kingman1976subadditive, steele1997probability}.

\begin{theorem}
	\label{thm:FK for Birkhoff}
	Let $(A_n)$ be a strictly stationary, ergodic stochastic process of $p \times p$ matrices satisfying the conditions of Theorem \ref{thm:barMnx vs barMnw}. Then we have   
	$$
	\lim_{m \rightarrow \infty} \frac{1}{m} \log\tau(M_m) =  \lim_{m \rightarrow \infty} \frac{1}{m} \mE \log\tau(M_m) \quad {\rm w.p.1}.
	$$
\end{theorem}
\begin{proof} [Proof of Theorem \ref{thm:FK for Birkhoff}.] 
The double index series $M_{m,k} = A_m  A_{m-1} \cdots A_{k}$ is obviously strictly stationary, $M_{m+1,k+1}(\omega)=  M_{m,k}(T \omega),$ where $T$ is ergodic. It follows that the double index series $\log \tau(M_{m,k})$ is also strictly stationary. Moreover, it is obviously sub-additive, and $\mE \log^+ \tau (M_{1,1} ) = 0$ since $ \tau (M_{1,1}) \le 1.$ Thus by the sub-additive ergodic theorem we have 
$$
\lim_{m \rightarrow \infty} {\frac 1 {m}} \log \tau(M_{m,1}) = \lim_{m \rightarrow \infty} \frac{1}{m} \mE
 \log\tau(M_{m,1}) \quad {\rm w.p.1},  
$$
which proves our claim. 
\end{proof}

Combining this theorem with Corollary \ref{cor:gap vs lim Birkhoff} we get the following extension:

\begin{corollary}
  \label{cor:gap ge a.s. lim av log Birkhoff Mm}
  Let $(A_n)$ be a strictly stationary, ergodic stochastic process of $p \times p$ matrices satisfying the conditions of Theorem \ref{thm:barMnx vs barMnw}. Then we have the following lower bound for the spectral gap: 
  \begin{equation}
    \label{eq:gap ge a.s. lim av log Birkhoff Mm}
    \lambda_1-\lambda_2 \ge   \lim_{m \rightarrow \infty} - \frac{1}{m} \log\tau(M_m) \quad {\rm w.p.1}.
  \end{equation}
\end{corollary}

The above results can be interpreted also as lower bounds for $\log \tau(M_m)$ in various forms. We will now develop an almost sure upper bound for $\log \tau(M_m)$ using the techniques developed in the previous sections. Taking into account \eqref{eq:Birkhoff vs tanh} the Birkhoff contraction coefficient $\tau(M_m),$ for its small values and for $M_m >0$, is equivalent to $\varphi(M_m).$ On the other hand, $\varphi(M_m)$ is a measure of collinearity of the rows of $\tau(M_m),$ see \eqref{eq:def varphi A}. Thus an upper bound for $\tau(M_m)$ provides a bound
on the speed with which $M_m$ converges to a rank-$1$ matrix.

\begin{theorem}
  \label{thm:Birkhoff le -gap a.s.}
  Assume that any of the sets of conditions of Theorems \ref{thm:ei Mnx vs ei Mnw IID}, \ref{thm:ei Mnx vs ei Mnw STATIONARY BOUNDED} or \ref{thm:ei Mnx vs ei Mnw STATIONARY GEN} is satisfied. Then we have 
  \begin{equation}
    \label{eq:Birkhoff le -gap a.s.}
    \limsup_{n \rightarrow \infty} {\frac 1 n} \log \tau(M_n) \le -(\lambda_1 - \lambda_2) \quad {\rm w.p.1.} 
  \end{equation}
\end{theorem} 

\begin{proof} [Proof of Theorem \ref{thm:Birkhoff le -gap a.s.}.]
The conditions of the theorem are identical to those of Lemma \ref{lem:Mnik vs Mnjk subexp}, implying that for any pair of row indices $i,j$ and any column index $k$ the quotient ${M_{n}^{ik}}/{M_{n}^{jk}}$ is sub-exponential, and thus Condition \ref{cond:Mnik vs Mnjk subexp} is satisfied. It follows that the conditions of Lemma \ref{lem:1 over u1n} are also satisfied, implying that $1/u^{i 1}_{n}$ is sub-exponential a.s. for all $i.$

 Now, consider the equality (\ref{eq:Mn ik vs jk FNL}), developed in the course of the proof of Lemma \ref{lem:1 over u1n}. Recall that $|u^{j 1}_{n}| \le 1$ and $1 /u^{i 1}_{n}$ is sub-exponential for all $i,j.$ Hence dividing both the numerator and the denominator of \eqref{eq:Mn ik vs jk FNL} by $u^{i 1}_{n},$ we get, independently of the column index $k,$ 
\begin{equation}
\label{eq:Mn ik vs jk V2}
\frac {M_{n}^{jk}} {M_{n}^{ik}} =  \frac {u^{j 1}_{n}} {u^{i 1}_{n}} +  O(e^{(-\lambda_1 + \lambda_2 + o(1))n}) \quad {\rm a.s.}
\end{equation}
By assumption for sufficiently large (random) $n,$ the matrix $M_n$ is strictly positive, hence we can write, see \eqref{eq:def varphi A}, 	
\begin{equation}
\varphi (M_n) = \max_{i,j,k,l}  \log \left( \frac {M_{n}^{jl}} {M_{n}^{il}} \right) / \left(  \frac {M_{n}^{jk}} {M_{n}^{ik}}  \right).
\end{equation}	
Taking into account \eqref{eq:Mn ik vs jk V2}, and once again noting that $|u^{j 1}_{n}| \le 1$ and $1 /u^{i 1}_{n}$ is sub-exponential for all $i,j,$ we get a.s. 
\begin{equation*}
\varphi(M_n)  =  O (\log (1 + e^{(-\lambda_1 + \lambda_2 + o(1))n})) =   O (e^{(-\lambda_1 +\lambda_2 + o(1))n}). 
\end{equation*}
Taking into account Birkhoff's quoted theorem, stating that $ \tau(M_n) = \tanh \left( {\varphi(M_n)} /4 \right),$
we immediately get 
\begin{equation}
\tau(M_n) =  O (e^{(-\lambda_1 +\lambda_2 + o(1))n}),
\end{equation}
from which the theorem immediately follows.
\end{proof}

From the theorem above we get via a trivial rearrangement an a.s. upper bound for the spectral gap in terms of Birkhoff contraction coefficient: 
	\begin{equation}
	\label{eq:gap le -Birkhoff a.s.}
	 \lambda_1 - \lambda_2 \le  - \limsup_{n \rightarrow \infty} {\frac 1 m} \log \tau(M_m) \quad {\rm w.p.1.}
	 \end{equation}
We have seen that on the right hand side $\limsup$ can be replaced with $\lim.$ Combining the above upper bound for the gap with the lower bound obtained in Corollary \ref{cor:gap ge a.s. lim av log Birkhoff Mm} we get the following result:

\begin{theorem}
	\label{thm:gap eq Birkhoff asymp} 
	Assume that any of the sets of conditions of Theorems \ref{thm:ei Mnx vs ei Mnw IID}, \ref{thm:ei Mnx vs ei Mnw STATIONARY BOUNDED} or \ref{thm:ei Mnx vs ei Mnw STATIONARY GEN} is satisfied. Then we have  
	\begin{equation}
	\lambda_1 - \lambda_2 =  \lim_{m \rightarrow \infty} - {\frac 1 m} \log \tau(M_m) \quad {\rm w.p.1.} 
	\end{equation}
\end{theorem}

  \begin{remark}
  We note in passing that a straightforward extension of (\ref{eq:ek Mnx vs ek Mnw fnl}) yields the following: let $u, v \ge 0$ be non-zero vectors, then we have a.s. 
	\begin{equation*}
          \left (\frac{u^\top M_{n} x}{u^\top M_{n} w} \right)/
	\left (\frac{v^\top M_{n} x}{v^\top M_{n} w} \right)
	= 1 +  O(e^{(\lambda_2 -\lambda_1 + o(1))n}).
	\end{equation*}
	In the case when we take a fixed non-negative, allowable, primitive matrix $A,$ we easily get the following result: for all pairs of non-negative, non-zero vectors $(u, v),$ except for a set of Lebesgue-measure zero, we have a.s. 
	\begin{equation*}
	\label{eq:tilde x Anx vs tilde w Anw double}
	\lim_{n \rightarrow \infty} {\frac 1 n} \log \log \left (\frac{u^\top A^{n} x}{u^\top A^{n} w} \right)/
	\left (\frac{v^\top A^{n} x}{v^\top A^{n} w} \right)
	= -(\lambda_1 - \lambda_2).
	\end{equation*}
  \end{remark}

\section{Discussion and conclusion }
\label{sec:Discussion and conclusion}

We should point out that the characterization of the a.s. rate of convergence via the spectral gap  $\lambda_1 - \lambda_2$ may provide a solid ground for further investigations of direct practical interest, such as explicit estimates on the relation of spectral gap with respect to the number of nodes, the failure probabilities or the strength of connectivity, see \cite{iutzeler2013analysis} on related empirical results to this effect. 
Let us mention two simple facts that may be relevant in such investigations.

First, we note $\lambda_1 ({\cal A})$ is monotone non-decreasing in ${\cal A}$. More precisely, letting ${\cal A} = (A_n)$ and ${\cal A'} = (A'_n)$, and assuming $A_n \le A'_n$ entry-wise for all $n$ w.p.1 implies $\lambda_1({\cal A}) \le \lambda_1 ({\cal A'})$. Indeed, $ A'_n A'_{n-1}\cdots A'_1$ is entry-wise not less than $A_n A_{n-1}\cdots A_1,$ hence letting $\| B \| = \sum_{i,j} b_{ij},$ we have $\|A_n A_{n-1}\cdots A_1\|  \le \| A'_n A'_{n-1}\cdots A'_1 \|, $ implying the stated inequality. From the above observation we immediately get the following simple result: 

\begin{lemma}
Let $(A_n)$ and $(A'_n)$ be two strictly stationary, ergodic processes of matrices associated with the push-sum method on the same underlying network but with with packet loss probabilities $r_{ij} \le r_{ij}^{ \prime}$ for all $i,j.$ Then $\lambda_1({\cal A}) \ge \lambda_1 ({\cal A'})$.
\end{lemma}

Unfortunately, the effect of increasing the packet loss probabilities on $\lambda_2$ is yet unknown. If we had $\lambda_2({\cal A}) \le \lambda_2({\cal A'})$ then we could conclude that increasing the packet loss probabilities would decrease, or at least not increase the gap. A nice observation here is that although we do not know if $\lambda_2({\cal A}) \le \lambda_2({\cal A'})$ we do know that $\sum_{i=2}^p \lambda_i({\cal A}) \le \sum_{i=2}^p \lambda_i({\cal A'}).$ The last inequality follows from a simple relationship for the sum of the Lyapunov-exponents given in the lemma below:  

\begin{lemma}
\label{lem:sum of Lyap exps}
Let $(A_n)$ be a sequence of $p \times p$ matrices satisfying the conditions of Proposition \ref{prop:FK}. Then we have 
$$
\lambda_1 + \ldots + \lambda_p = {\mE} \log ( | \det A_1 | )  
$$
In the case of the push-sum algorithm allowing packet loss we get $\lambda_1 + \ldots + \lambda_p = - \log 2.$
\end{lemma}

The magic of the lemma is that the l.h.s. depends only on the marginal distribution of $A_1$. 

\begin{proof} [Proof of Lemma \ref{lem:sum of Lyap exps}]
For the $p$-factor exterior product we have, 
$$
A_n \wedge \cdots \wedge A_n = \det A_n.
$$
Therefore
$$
\Pi_{k=1}^n \left( A_k \wedge \cdots \wedge A_k \right) =  \Pi_{k=1}^n \det A_k. 
$$
On the other hand, using the singular value decomposition $A_n \cdots  A_1 = U_n \Sigma_n V_n$ we can write
\begin{align*}
&\Pi_{k=1}^n \left(A_k \wedge \cdots \wedge A_k  \right) \nonumber \\
= &\Pi_{k=1}^n \left( U_k \wedge \cdots \wedge U_k \right) \cdot \Pi_{k=1}^n \left( \Sigma_k \wedge \cdots \wedge \Sigma_k\right) \nonumber \\ 
\cdot & \Pi_{k=1}^n \left( V_k \wedge   \cdots \wedge V_k \right) 
\end{align*}
Therefore 
$$
\Pi_{k=1}^n \det A_k = \pm \Pi_{k=1}^n \det \Sigma_k =  \pm \Pi_{k=1}^n \sigma^{1}_{k} \cdots \sigma^{p}_{k}.
$$
Taking absolute value and logarithm, dividing by $n,$ and the taking the limit, we get
$$
{\mE} \log ( | \det A_1 | ) =  \lambda_1 + \ldots + \lambda_p.
$$

In the case of the push-sum algorithm allowing packet loss we have $| \det A_n | = 1/2$ for all $n$ and all $\omega$, thus we get the claim. 
\end{proof}

\begin{remark}
  Setting $p=2$ the combination of the above observations give that in the case of the push-sum algorithm increasing the probabilities of packet loss will decrease the spectral gap:
\end{remark}

\begin{equation}
\lambda_1({\cal A}) -  \lambda_2({\cal A})  \ge \lambda_1 ({\cal A'}) - \lambda_2 ({\cal A'})
\end{equation}
for any strictly stationary, ergodic $2 \times 2$ matrix-valued processes $(A_n)$ and $(A'_n)$ of the form \ref{eq:def An PS with Loss}, no matter what the dependence structure is.

\begin{remark}
Finally we should note in retrospect that Theorem \ref{thm:gap for seq primitive proc} implies that the conditions $\lambda_1 - \lambda_2$ in our main results Theorems \ref{thm:ei Mnx vs ei Mnw IID}, \ref{thm:ei Mnx vs ei Mnw STATIONARY BOUNDED}, \ref{thm:ei Mnx vs ei Mnw STATIONARY GEN} and \ref{thm:ei Mnx vs ei Mnw COLUMN STOCH} can be removed, namely it is implied by the assumption that $(A_n)$ is sequentially primitive. Similarly, in the case of the push-sum algorithm, Theorem \ref{thm:ei Mnx vs ei Mnw PS}, the claim that $-(\lambda_1 - \lambda_2) < 0$ follows immediately from Theorem \ref{thm:gap for seq primitive proc}.
\end{remark}

  This observation combined with Theorem \ref{thm:barMnx vs barMnw} has the following nice implication. Let $x, w$ be probability vectors. The $x_n, w_n$ will be probability vectors for all $n,$ which can be interpreted as the distributions generated by a finite-state Markov-chain in a random, strictly stationary environment with initial distributions $x, w$.
Then the statement of Theorem \ref{thm:barMnx vs barMnw}, with the assumption $\lambda_2 < 0$ ensured by Theorem \ref{thm:gap for seq primitive proc} as a consequence of sequential primitivity, specializes to 
\begin{equation*}
  \limsup_{n \rightarrow \infty} {\frac 1 n} \log \left \Vert {x_n} - {w_n} \right \Vert_{\rm TV} = \lambda_2 < 0, 
\end{equation*}
for almost all $(x,w)$
stating a kind of exponential stability for Markov-chains in random environment. Furthermore, combining with Theorem \ref{thm:ei Mnx vs ei Mnw COLUMN STOCH} we get a rate of stability for ratios inspired by the separation distance
\begin{equation*}
  \limsup_{n \rightarrow \infty} {\frac 1 n} \log \max_i\left \vert \frac{x^i_n}{w^i_n}  - 1 \right \vert = \lambda_2 < 0,
\end{equation*}
once again for almost all $(x,w)$ initial distributions.

{\emph {Potential connections}.} We thank to our anonymous reviewers for calling our attention to papers that may be relevant to the problems discussed above, such as \cite{hadjicostis2011asynchronous}, and the follow-up paper \cite{hadjicostis2014average} developing a ratio consensus algorithm allowing arbitrary bounded delays. In fact, the results of our paper, combined with the basic ideas of \cite{hadjicostis2014average}, are directly applicable to this class of problems. Secondly, an ingenious device was proposed in \cite{silvestre2018broadcast}, using auxiliary variables to solve the average consensus problem with column stochastic matrices via a \emph{linear} asynchronous gossip algorithm, proving exponential mean square stability with an explicit upper bound for the rate. Our results seem to be applicable to prove almost sure exponential convergence, the rate of which is superior to the rate provided by \cite{silvestre2018broadcast} due to a simple convexity argument.

{\emph {Conclusion}.} The problems discussed in the paper are motivated by the ratio consensus problems and algorithms, such as the push-sum or weighted gossip algorithms. We have considered fairly general, strictly stationary  communication protocols, covering as special cases broadcast algorithms, geographic gossip, randomized path averaging or one-way averaging. We have given sharp upper bounds for the rate of almost sure exponential convergence in terms of the spectral gap of the associated matrix sequence under various technical conditions. We have presented a variety of connections between the spectral gap and the Birkhoff contraction coefficient of the product of the associated matrices. Our results significantly extend relevant results of \cite{iutzeler2013analysis}, and provide a solution to an open problem raised in \cite{benezit2010weighted}.

\section{Appendix I. Sequential primitivity}
\label{sec:App I Seq primitivity}

Lemma \ref{lem:seq prim forward backward} is a direct consequence of the lemma below, a standard device in queuing theory:

\begin{lemma}
  \label{lem:forward backward paradox} 
  Let $(\Delta_n)$ be a two-sided strictly stationary, non-negative process. Define for all $n$
  \begin{equation}
    m_n = \max_{m \le n} \{m + \Delta_m \le n  \} \qquad {\rm and} \qquad  \Delta'_n = n - m_n.
  \end{equation}
  Then the distributions of $\Delta_n$ and $\Delta'_n$ are the same for all $n.$ In particular, $\mE \Delta_n = \mE \Delta'_n.$
\end{lemma}

\begin{proof} [Proof of Lemma \ref{lem:forward backward paradox}.]
  We have for any $x \ge 0$ 
  \begin{align}
    P(\Delta'_n > x)= P(n - m_n > x)= P(m_n < n- x)\nonumber \\ 
    = P(n- x + \Delta_{n- x} > n)= P(\Delta_{n- x} > x).
  \end{align}
  Since $(\Delta_n)$ is strictly stationary we have $P(\Delta_{n- x} > x)=
  P(\Delta_{n} > x),$ as claimed.
\end{proof}

The lemma above describes an apparent paradox between forward and backward waiting times, since at any time $n$ we have $\Delta'_n  \ge \Delta_{m_n},$ and this may tempt us to believe that $\Delta'_n$ is stochastically larger than $\Delta_n,$ which would contradict to the symmetry between forward and backward.

\begin{proof} [Proof of Lemma \ref{lem:sigma n tails iid A n}] Let the elements of ${{\cal A}}$ be denoted by $B_1, B_2, \ldots ,B_r$ so that $P(A_1 = B_i) > 0$ for all $i.$ The i.i.d. sequence $(A_n)$ can be identified with an i.i.d.\ sequence of indices $i_{1},i_{2}, \ldots,$ with $1 \le i_k \le r$.
  Since $\cal A$ is primitive, there exists a word $w = (j_s,j_{s-1}, \ldots, j_1)$ such that $B_{j_s} B_{j_{s-1}} \cdots B_{j_1} >0.$
  Segment the full sequence of indices into an i.i.d. sequence of $s$-tuples $v_m.$ Let $\tau :=\min \{m: v_m = w\}.$ Since $p:=P(v_m = w)>0$ implies $P(\tau > x) = (1 - p)^x $ and $\psi_1 \le m \tau$, the claim follows. 	

\end{proof}

\section{Appendix II. Normalized products}
\label{sec:App II Normalized products Correct AZ}

In this section we present the proof of Theorem \ref{thm:barMnx vs barMnw}, starting with the proofs of the auxiliary results, Lemma \ref{lem:Oseledec for An positive and x positive} and \ref{lem:almost all x and w pairs}.

\begin{proof} [Proof of Lemma \ref{lem:Oseledec for An positive and x positive}] Consider 
	$$
	M_n^\top M_n = V_n^\top {\rm diag} (2 \sigma^{i}_n) V_n.
	$$
	For $n \ge \tau$ this is a symmetric positive semi-definite matrix with strictly positive elements. Its eigenvalues are $2 \sigma^{i}_n$ with corresponding eigenvectors $(v^{i\cdot}_n)^\top$. By the Perron-Frobenius theorem $M_n^\top M_n$ has a unique eigenvalue with maximal modulus, which is positive as is the corresponding eigenvector. It follows that $2 \sigma^{1}_{n}$ is a single eigenvalue, and $v^{1\cdot}_n > 0$ elementwise.

	Expand $x$ in the orthonormal system defined by the rows of $V_n$: 
	$x^\top = \sum \alpha^{i}_n v^{i\cdot}_n.$ Here $\alpha^{i}_n := v^{i\cdot}_n x$. Then 
	$$ 
	x^T M_n^\top M_n x =  \sum (\sigma^{i}_n)^2 (\alpha^{i}_n)^2 .
	$$
	Now, $v^{1\cdot}_n > 0$ and $\vert v^{1\cdot}_n \vert =1,$ together with $x >0$ imply that $\alpha_{1n} > \alpha_1 >0$ with some $\alpha_1.$ Thus $x^T M_n^\top M_n x >  (\sigma^{1}_{n})^2 \alpha_1^2,$ from which we get 
	$\liminf {\frac 1 n} \log \vert x^T M_n^\top M_n x \vert \ge 2 \lambda_1,$ implying $\liminf {\frac 1 n} \log \vert M_n x \vert \ge \lambda_1,$ and thus the claim of the lemma follows. 
\end{proof}

\begin{proof} [Proof of Lemma \ref{lem:almost all x and w pairs}] Write $V_1' = \mathbb R^p \wedge \mathbb R^p.$ According to Oseledec's theorem there is a proper random subspace of
	$V_1'$ of fixed dimension, say $V_2',$ such that for $z\in V_1' \setminus V_2'$
	\begin{align*}
	 \lim_{n \rightarrow \infty} & \frac{1}{n} \log \vert ((A_nA_{n-1}\cdots A_1) \wedge (A_nA_{n-1}\cdots A_1))z \vert \nonumber \\  = & \lambda_1 + \lambda_2 \quad {\rm a.s.}
	\end{align*}
	Consider the tensor product space $\mathbb R^p \otimes \mathbb R^p$ and
	its canonical linear mapping to $V_1' = \mathbb R^p \wedge \mathbb R^p$, denoted by $S,$ defined by
	$$
	\sum_{i,j} x_{ij} ~e_i \otimes e_j {\longrightarrow} \sum_{i,j}
	x_{ij} e_i \wedge e_j = \sum_{i<j} (x_{ij}-x_{ji}) e_i \wedge e_j.
	$$
	Equivalently, interpreting $\mathbb R^p \otimes \mathbb R^p$ as the linear space of matrices of size $p\times p,$ and identifying $\mathbb R^p \wedge \mathbb R^p$ as the linear space of antisymmetric matrices, the linear transformation $S$ takes the form $S(X) = X - X^\top$.

	It is readily seen that $V_2" = S^{-1} V_2'$ is a proper subspace of the tensor product space $\mathbb R^p \otimes \mathbb R^p$. Indeed, any $X \in \mathbb R^p \otimes \mathbb R^p$ can be written as $X=X_a+X_s$, as a sum of its antisymmetric
	and symmetric part, and we have $S(X) = 2 X_a$. Therefore the linear subspace $V_2" = S^{-1} V_2'$ consists of matrices for which $X_a \in V_2'$, and thus it is indeed a proper subspace.

	Let $E$ denote the random set of exceptional pairs $(x,w)(\omega)$ defined as  
	\begin{equation}
	E_{xw}(\omega) = \{(x,w): x \otimes w \in V_2"(\omega) \}
	\end{equation}
	We claim that $E_{xw}(\omega) \in \mathbb R^p \times \mathbb R^p$ has zero Lebesgue-measure for all almost all $\omega.$ Assuming the contrary, there is a set $E_x(\omega) \in \mathbb R^p$ of positive Lebesgue measure
	such that for each $x\in E_x(\omega)$ the set 
	$$
	E_{w|x}(\omega) =  \{w: (x,w) \in E_{xw}(\omega) \} 
	$$
	has positive Lebesgue-measure in $\mathbb R^p.$ Taking any $x \in E_x(\omega)$, the elements of $E_{w|x}(\omega)$ span the full $\mathbb R^p$,
	therefore $(x,w), ~w \in E_{w|x}(\omega)$ span the linear space $x \otimes \mathbb R^p$. Letting $x$ vary through the positive set $E_x(\omega)$ we get that the elements of $x \otimes \mathbb R^p$ span the whole $\mathbb R^p \times \mathbb R^p.$ 
	This is in contradiction with the assumption any for $(x,w) \in E_{xw}(\omega)$ the tensor product $x \otimes w$ lies in the proper subspace $V_2".$

	We conclude by Fubini's theorem that the exceptional set in $\mathbb R^p \times \mathbb R^p \times \Omega$ 
	\begin{equation}
	E_{xw\omega} = \{(x,w,\omega): (x,w) \in V_2"(\omega)\}
	\end{equation}
	has $\lambda \times \lambda \times P$-measure zero. Applying Fubini's theorem once again in the opposite direction we get the claim. 
\end{proof}

\begin{proof}[ Proof of Theorem \ref{thm:barMnx vs barMnw}]
	Note that since $\bar x_n$ and $\bar w_n$ belong to the simplex of probability vectors we have  
		\begin{equation*}
		\| \bar x_n - \bar w_n \|_{\rm TV} \sim \vert \sin (\bar x_n, \bar w_n)\vert =   \vert \sin (x_n, w_n)\vert = \frac {\vert  x_n \wedge  w_n \vert}   {\vert   x_n \vert   \cdot \vert  w_n \vert},
		\end{equation*}
		where $a_n \sim b_n$ means that $a_n / b_n$ and $b_n / a_n$ are bounded by a deterministic constant. After taking logarithm we get that $\log \| \bar x_n - \bar w_n \|_{\rm TV}$ can be written as 
		\begin{align}
		\log {\vert  x_n \wedge  w_n \vert} - \log \vert   x_n \vert  - \log \vert  w_n \vert + O(1), 
		\label{eq:barxn vs barwn appr by vectorial}
		\end{align}
		where $O(1)$ is bounded by a deterministic constant.

	To deal with the second and third terms of \eqref{eq:barxn vs barwn appr by vectorial} we use Lemma \ref{lem:Oseledec for An positive and x positive}, from which we get for any strictly positive initial vectors $x,w >0$ almost surely
		\begin{equation}
		\lim_{n \rightarrow \infty} {\frac 1 n} \log \vert x_n \vert = \lambda_1 \quad {\rm and } \quad 	\lim_{n \rightarrow \infty} {\frac 1 n} \log \vert w_n \vert = \lambda_1. 
		\end{equation}

	To deal with the first term of \eqref{eq:barxn vs barwn appr by vectorial} we use Lemma \ref{lem:almost all x and w pairs} implying that for all initial pairs $(x,w) \in \mathbb R^p_+ \times \mathbb R^p_+$, except for a set of Lebesgue measure zero, we have 
	\begin{equation}
	\label{eq:xn wedge wn abs}
	\lim_{n \rightarrow \infty} \frac{1}{n} \log \vert (x_n \wedge w_n) \vert = \lambda_1 + \lambda_2 \quad {\rm a.s.}
	\end{equation}
	Moreover, Oseledec's theorem implies that \emph{for all} initial pairs $(x,w) \in \mathbb R^p \times \mathbb R^p$ the left hand side of (\ref{eq:xn wedge wn abs}) exists, and it is majorized by the right hand side w.p.1. Combining these facts with (\ref{eq:barxn vs barwn appr by vectorial}) we immediately get Theorem \ref{thm:barMnx vs barMnw}.
\end{proof}

{\section{Appendix III. The push-sum algorithm}
	\label{sec:Appendix  III. The push-sum algorithm}

\medskip
\begin{proof} [Proof of Lemma \ref{lem:cal APS prim}] 	The basic idea is what is called flooding. A convenient reference is Lemma 4.2 of \cite{benezit2010weighted}, the conditions of which can be readily verified, implying that there exist an $N$ such that $p: = P(A_N \cdots A_1 > 0) >0,$ where the strict inequality is meant entry-wise. It follows that for any $m \ge 1$ we have $P(A_{mN} \cdots A_1 > 0) > 1 - (1 - p)^m,$ and the claim follows by a Borell-Cantelli argument. 

\end{proof}

	\begin{proof} [Proof of Lemma \ref{lem:gap for PS is positive}]
		The following proof relies on a combination of \cite{gerencser2018push} and Theorem \ref{thm:barMnx vs barMnw}. Note that our conditions are identical with those of \cite{gerencser2018push}, except that in there $\alpha_{ji}=1/2$ for all $(j,i) \in E$ and $w = \mathbf 1$ were assumed. It is easily seen that the analysis of Theorem 3 in \cite{gerencser2018push} carries over for general $w \ge 0, w \neq 0$ and $\alpha_{ji}\in (0,1).$ In particular, setting $s_n = \mathbf 1^\top x_n$ and $t_n = \mathbf 1^\top w_n$, we get by a straightforward extension of Theorem 3 in \cite{gerencser2018push} : 
		for any vector of initial values $x \in  \mathbb R^p,$ and a non-negative vector of initial weights $w \in \mathbb R^p_+ $ such that $w \neq 0$ we have for all $~i=1, \ldots p$ ~a.s.
		\begin{equation}
		\lim_{n \rightarrow \infty} \frac{{x}_{ni}}{{w}_{ni}} =
		\lim_{n \rightarrow \infty} \frac{s_n}{t_n} \cdot \frac{\bar{x}_{ni}}{\bar{w}_{ni}}
		= x^*  
		\end{equation}
		for some random $x^*.$ In fact the convergence is at least exponential with a deterministic
		rate: for all $~i=1, \ldots p$ ~a.s.
		\begin{equation}
		\label{eq:PUSH SUM EXP RATE}
		\frac{s_n}{t_n} \cdot \frac{\bar{x}_{ni}}{\bar{w}_{ni}} =
		x^* + O(e^{-\alpha n}).
		\end{equation}
		It follows by a simple convexity argument (see the proof of Corollary \ref{cor:q Mnx vs q Mnw}) that we also have $s_n/t_n \ra x^*$ a.s. exponentially fast with the same rate:
		\begin{equation}
		\label{eq:PUSH SUM MEAN EXP RATE}
		\frac{s_n}{t_n} = x^* + O(e^{-\alpha n}) \quad  ~{\rm a.s.}
		\end{equation}
		In addition, $x^*$ is a convex combination of the initial ratios $x_k/w_k$. It follows that choosing $x, w>0$ we will have $x^* >0.$

		Hence dividing (\ref{eq:PUSH SUM EXP RATE}) by (\ref{eq:PUSH SUM MEAN EXP RATE}) we get
		for all $~i=1, \ldots p$ ~a.s.
		\begin{equation}
		\label{eq:bar xk over bar wk }
		\frac{\bar{x}_{ni}}{\bar{w}_{ni}} = 1 + O(e^{-\alpha n}). 
		\end{equation}
		From this the exponential decay of the total variation distance of $\bar{x}_{ni}$ and $\bar{w}_{ni}$ immediately follows:
		multiplying both sides of (\ref{eq:bar xk over bar wk }) by $0<\bar{w}_{ni} 
		\le \max_i w_i,$ followed by summation over $i$ gives the almost sure asymptotics  
		\begin{equation*}
		\vert \bar{x}_{ni} - \bar{w}_{ni} \vert = O(e^{-\alpha n}) \quad \text{and}  \quad \|\bar{x}_n - \bar{w}_n \|_{\rm TV} = O(e^{-\alpha n}),
		\end{equation*}
		and hence for all strictly positive pairs $(x,w)$ we get 
		\begin{equation}
		\limsup_{n \rightarrow \infty} \frac{1}{n} \log \|\bar{x}_n - \bar{w}_n\|_{\rm TV} < 0 \quad {\rm a.s.}
		\end{equation}
		But the left hand side is equal to $-(\lambda_1 - \lambda_2)$ a.s. for Lebesgue-almost all $(x,w) \in \mathbb R^p_+ \times  \mathbb R^p_+ ~x,w \neq 0$ by Theorem \ref{thm:barMnx vs barMnw} with $\limsup$ replaced by $\lim.$ Thus 
	\end{proof}

\section{Appendix IV. ${M_{n}^{ik}}/{M_{n}^{jk}}$ is sub-exponential}
\label{sec:Appendix IV. M_n ik / M_n jk is sub-exponential}

We first provide an elementary a priori estimate of ${M_{n}^{ik}} / {M_{n}^{jk}}$ using using the following lemma, a variant of which has been stated by Bellman, see \cite{furstenberg1960products}, \cite{bellman1954limit}. 

\begin{lemma} 
	\label{lem:M=XB}
	Let $M, B, X$ be $p \times p$ matrices such that $M= BX.$ Assume that $B$ is strictly positive, and $X$ is a non-negative, allowable matrix. Then $M$ is strictly positive, and for any fixed pair of row indices $(i,j)$ and any column index $k$ we have 
	$$
	\min_r \frac {B_{ir}} {B_{jr}} \le ~\frac {M^{ik}} {M^{jk}} 
	~\le \max_r \frac {B_{ir}} {B_{jr}}
	$$  
\end{lemma}

\begin{proof} [Proof of Lemma \ref{lem:M=XB}]
	The $(i,k)$ and the $(j,k)$ element of $M$ can be expressed as
	$$
	M^{ik} = \sum_{r} B_{ir} X_{rk}  \quad {\rm and} \quad M^{jk} = \sum_{r} B_{jr} X_{rk}.
	$$
	It is easily seen that the ratio $M^{ik}/M^{jk}$, i.e.
	$$
	\frac {M_{i,k}}{M^{jk}} = \frac{\sum_{r} B_{ir} X_{rk}}{\sum_{r} B_{jr} X_{rk}}
	$$
	can be written as a convex combination of $B_{ir}/B_{jr}$ with weights 
	$$
	\mu_r = B_{jr} X_{rk}/{\sum_{s} B_{js} X_{sk}}
	$$
	which implies the claim.
\end{proof}

\begin{lemma}
  \label{lem:Mnik vs Mnjk subexp} Under any set of conditions given in Theorem \ref{thm:ei Mnx vs ei Mnw IID}, \ref{thm:ei Mnx vs ei Mnw STATIONARY BOUNDED}, \ref{thm:ei Mnx vs ei Mnw STATIONARY GEN} it holds that for any pair of row indices $i,j$ and any column index $k$ the quotient ${M_{n}^{ik}}/{M_{n}^{jk}}$ is sub-exponential. 
\end{lemma}

\begin{proof}[Proof of Lemma \ref{lem:Mnik vs Mnjk subexp}] 
	In order to apply Lemma \ref{lem:M=XB} let us first extend the sequence $(A_n)$ for $n \le 0,$ with eventual extension of the underlying probability space, so that we get a two-sided strictly stationary, ergodic sequence, or even i.i.d.\ sequence  in the case of Theorem \ref{thm:ei Mnx vs ei Mnw IID}. Recall the definition of the index of backward sequential primitivity: 
		$$
		\rho_n  = \min \{\rho \ge 0: A_{n} A_{n-1} \cdots A_{n-\rho+1} >0  \}.
		$$
		Note that under any set of conditions given in Theorems \ref{thm:ei Mnx vs ei Mnw IID}, \ref{thm:ei Mnx vs ei Mnw STATIONARY BOUNDED}, \ref{thm:ei Mnx vs ei Mnw STATIONARY GEN} we can claim that $\mathbb{E} \rho_n < \infty.$ Indeed, under the conditions of Theorem \ref{thm:ei Mnx vs ei Mnw IID} $\mathbb{E} \rho_n < \infty$ follows from Lemma \ref{lem:sigma n tails iid A n}. On the other hand, $\mathbb{E} \rho_n < \infty$ follows from the condition  $\mathbb{E} \psi_n < \infty,$ that was a priori assumed to hold in the case of Theorems \ref{thm:ei Mnx vs ei Mnw STATIONARY BOUNDED} and \ref{thm:ei Mnx vs ei Mnw STATIONARY GEN}, due to Lemma \ref{lem:seq prim forward backward}. Consider now the sets 
		$$
		\Omega^G_{n} = \{\omega: 
		\rho_n \le n \} \quad {\rm and} \quad \Omega^{Gc}_{n} = \Omega \setminus \Omega^G_{n}. 
		$$
		Note that $\mE \rho_n <\infty$ implies that  
		$$
		\sum_{n=1}^{\infty} P(\Omega^{Gc}_{n}) = \sum_{n=1}^{\infty} (1- P(\Omega^G_n)) = \sum_{n=1}^{\infty} P(\rho_n > n) < \infty,
		$$
		and thus $\Omega^{Gc}_{n}$ occurs finitely many times w.p.1. by the Borel-Cantelli lemma. Equivalently, the set
		\begin{equation}
		\Omega^{Gc} := \limsup_{n \rightarrow \infty} ~\Omega^{Gc}_{n} = \bigcap_{m \ge 1} \bigcup_{n \ge m} \Omega^G_{n}
		\end{equation}
		has measure $0,$ and consequently its complement  
		\begin{equation}
		\Omega^G := \liminf_n ~\Omega^{G}_{n} = \bigcup_{m \ge 1} \bigcap_{n \ge m} \Omega^G_{n} 
		\end{equation}
		has probability $1.$
		On the set $\Omega^G_{n}$ consider the following decomposition of $M_n$ by separating a strictly positive factor $B_n$ on the left: 
		\begin{equation}
		M_n = A_n A_{n-1}\cdots A_{n-\rho_n+1} \tilde{M}_n = B_n \tilde{M}_n. 
		\end{equation}

	Let $\beta'_n = \sum_{k,l} A_{n}^{kl}.$ Obviously, $\beta'_n$ is equivalent to $\beta_n = \max_{k,l} A_{n}^{kl},$ and also to $\Vert A_n \Vert,$ i.e. $\beta'_n \sim \beta_n \sim \Vert A_n \Vert.$ Then, a simple crude estimator of $\min_r {B_{n}^{ir}}/{B_{n}^{jr}}$ can be obtained on the set $\Omega^G_n,$ with $\alpha_n$ defined under (\ref{eq:def alpha n beta n}), as follows:
		\begin{equation}
		\label{eq:Bni vs Bnj LUB}
		\frac  {\Pi_{m=n-\rho_n +1}^n \alpha_m} {\Pi_{m=n-\rho_n +1}^n \beta'_m}   \le ~\frac {B_{n}^{ir}} {B_{n}^{jr}} ~\le \frac  {\Pi_{m=n-\rho_n +1}^n \beta'_m} {\Pi_{m=n-\rho_n +1}^n \alpha_m}.  
		\end{equation}
		Obviously, the lower bound is the reciprocal of the upper bound. We will estimate the latter from above. From the inequality (\ref{eq:Bni vs Bnj LUB}) we get on $\Omega^G_n$
		\begin{equation}
		~\log^+ \frac {B_{n}^{ir}} {B_{n}^{jr}} ~ \le {\sum_{m=n-\rho_n +1}^n \log^+ \beta'_m} -  {\sum_{m=n-\rho_n +1}^n \log^- \alpha_m}=: \pi_n. 
		\end{equation}
        Note that the middle term, and thus $\pi_n,$ is actually well-defined on all $\Omega$ (since $m$ can take on negative values) and obviously their distributions are independent of $n$. Thus, if we prove 
        ${\mE} \pi_n < \infty,$ it will imply that $\pi_n$ is sub-linear on $\Omega,$ yielding that ${B_{n}^{ir}} / {B_{n}^{jr}}$ is sub-exponential a.s. on $\Omega^G$ for any pair $(i,j)$ and any $r.$ This, in combination with Lemma \ref{lem:M=XB} yields the proof of Lemma \ref{lem:Mnik vs Mnjk subexp}.

        {\bf Claim}: Under any set of conditions given in Theorems \ref{thm:ei Mnx vs ei Mnw IID}, \ref{thm:ei Mnx vs ei Mnw STATIONARY BOUNDED}, \ref{thm:ei Mnx vs ei Mnw STATIONARY GEN} it holds that ${\mE} \pi_n < \infty.$

		The proof for the case of Theorem \ref{thm:ei Mnx vs ei Mnw IID}. Note that $\rho_n$ is a stopping time for the backward process with finite expectation. In addition, ${\mE} \log^+ \beta'_n < \infty.$ Moreover ${\mE} \log^- \alpha_n > -\infty,$ by Condition (\ref{cond:E log alpha n}). Since $\log^+ \beta'_n$ and $\log^- \alpha_n$ form i.i.d.\ sequences we get by Wald's theorem 
		\begin{align}
                  &\mE ~ \left( {\sum_{m=n-\rho_n +1}^n \log^+ \beta'_m} -  {\sum_{m=n-\rho_n +1}^n \log^- \alpha_m} \right) \nonumber \\
                  = ~ &\mE \rho_n \cdot  {\mE} \log^+ \beta'_1 - {\mE} \rho_n \cdot {\mE} \log^- \alpha_1 < \infty.
		\end{align}

		The proof for the case of Theorem \ref{thm:ei Mnx vs ei Mnw STATIONARY BOUNDED}, in which the positive elements of $A_n$ are assumed to be bounded from below by a positive bound and from above, is trivial: we have 
		\begin{align} 
		\label{eq:log Bir vs Bjr Thm STATIONARY BOUNDED}
                  & {\mE} ~ \left( {\sum_{m=n-\rho_n +1}^n \log^+ \beta'_m} -  {\sum_{m=n-\rho_n +1}^n \log^- \alpha_m} \right) \nonumber \\
                  \le ~ & {\mE} ~\rho_n \cdot \log^+ (p^2 \beta) - {\mE}~\rho_n \cdot {\mE} \log^- \alpha < \infty.
		\end{align}

		Finally, consider the case of Theorem \ref{thm:ei Mnx vs ei Mnw STATIONARY GEN}, in which the positive elements of $A_n$ may spread all over ${\mathbb R}_+.$ Setting $\lambda := {\mE} \log^+ \beta'_n,$ and noting that $(\log^+ \beta'_n)$ is ergodic, the random variable defined by 
		\begin{equation}
		C_n(\omega, \varepsilon) = \max_{k \ge 0} \left(\sum_{m=n-k}^n (\log^+ \beta'_m -  \lambda - \varepsilon ) \right)^+
		\end{equation}
		is finite w.p.1. for any $\varepsilon>0.$ Obviously, we have 
		\begin{equation}
		\sum_{m=n-\rho_n +1}^n \log^+ \beta'_m \le C_n(\omega, \varepsilon) +  (\lambda + \varepsilon) \rho_n.
		\end{equation}
		We can proceed with the estimation of $\sum_{m=n-\rho_n+1}^n \log^- \alpha_m$ analogously. Under the conditions of Theorem \ref{thm:ei Mnx vs ei Mnw STATIONARY GEN} we have ${\mE} \rho_n < \infty$. Obviously, $ (C_n(\omega, \varepsilon))$ is a strictly stationary sequence, therefore to complete the proof of the Claim it is sufficient to prove that ${\mE} C_n(\omega, \varepsilon) < \infty.$ This follows directly from the lemma below:

	\begin{lemma}
		\label{lem:E Risk for M mixing} 
			Let $(\xi_k), k \ge 1$ be a strictly stationary, ergodic process such that ${\mE} \xi_k =: - c < 0. $ Define 
			\begin{equation}
			\eta= \max_{m \ge 1} \left(\sum_{k=1}^m \xi_k \right)^+.
			\end{equation}
			Assume that $(\xi_k)$ is $M$-mixing of order $q$ with some $q >4.$ Then 
			${\mE} \eta < \infty.$ 
	\end{lemma}
	
	\begin{proof} [Proof of Lemma \ref{lem:E Risk for M mixing}]
		For any $x \ge 0$ we have 
		\begin{align}
		\label{eq:P eta ge x}
		 P(\eta \ge x) \le & \sum_{m=1}^{\infty} P\left( \sum_{k=1}^m \xi_k  \ge x\right)  \nonumber \\ = & \sum_{m=1}^{\infty} P\left( \sum_{k=1}^m (\xi_k + c)  \ge x + mc \right).
		\end{align}
		The $m$-th term on the r.h.s. can be bounded from above by using Markov's inequality for the $q$-th absolute moment and the condition that $(\xi_k)$ is $M$-mixing of order $q$ as follows:
		\begin{align}
		\label{eq:P eta ge x bound term m}
		  \frac {C_q m^{q /2}} {(x + mc)^{q}}  = & \frac {C_q m^{q /2}} {c^q~(x/c + m)^{q}}  \le \frac {C_q ~(x/c + m)^{q /2}} {c^q~(x/c + m)^{q}} \nonumber \\ 
		= & \frac {C_q} {c^q ~(x/c + m)^{q/2}}
		\end{align}
		with some $q > 4.$ Thus the sum over $m$ on the r.h.s. of \eqref{eq:P eta ge x} can be majorized, by noting that the right hand sides of \eqref{eq:P eta ge x bound term m} are monotone decreasing, as follows: 
		\begin{align}
		& \sum_{m=1}^{\infty} \frac {C_q} {c^q ~(x/c + m)^{q/2}} \le \int_0^{\infty}\frac {C_q} {c^q~(x/c + t)^{q/2}} dt \nonumber \\ = &\int_{x/c}^{\infty} ~\frac {C_q} {c^q ~t^{q/2}} dt=  {\frac {C_q} {c^q ~(-q/2 +1)}  } \left(\frac x c \right)^{-q/2 + 1}.
		\end{align}
		Summing through the positive integers $x=n,$ and recalling that $q >4,$ we conclude that 
		\begin{equation}
		\label{eq:sum P eta ge n}
		\sum_{n=1}^{\infty} P(\eta \ge n) \le \sum_{n=1}^{\infty} {\frac {C_q} {c^q (-q/2 +1)}  } \left(\frac n c \right)^{-q/2 + 1} < \infty,
		\end{equation}
		hence ${\mE} \eta < \infty,$ as stated in the lemma.
	\end{proof}
	
		It follows immediately, that the process 
		\begin{equation}
		\eta_n =  \max_{m \ge n} \left(\sum_{k=n}^m \xi_i \right)^+.
		\end{equation}
		is sub-linear. If $(\xi_i)$ is a two-sided process the same argument applies for the time-reversed process
		\begin{equation}
		\eta^r_n :=  \max_{m \le n} \left(\sum_{k=m}^n \xi_i \right)^+.
		\end{equation}
		With this the proof of Lemma \ref{lem:Mnik vs Mnjk subexp} is complete.
\end{proof}

{\bf Acknowledgment.} The first author expresses his thanks to Asuman Ozdaglar for encouraging him to study the push-sum algorithm with packet loss, and to Julien M. Hendrickx for an ongoing inspiring collaboration in consensus problems in general.

\bibliographystyle{ieeetr}
\bibliography{PushSum}

\end{document}